\documentclass[12pt]{amsart}
 \usepackage{amssymb, amsthm, amsmath, mathrsfs, graphicx, color}
\usepackage[all]{xy}

\usepackage{graphicx}
\usepackage[numbers,sort&compress]{natbib}
\usepackage{hypernat}
\usepackage{tikz}

\usepackage[utf8]{inputenc}
\usepackage[T1]{fontenc}

\usepackage{glossaries}

\usepackage{imakeidx}
\usepackage{hyperref}
\makeindex[title=Index of definitions,columns=2]
\makeindex[name=nota, title=Index of notation, columns=3]

 \makeindex

\DeclareSymbolFont{largesymbols}{OMX}{yhex}{m}{n}
\DeclareMathAccent{\widehat}{\mathord}{largesymbols}{"62}

\newtheorem{theorem}{Theorem}[section]
\newtheorem{corollary}[theorem]{Corollary}
\newtheorem{lemma}[theorem]{Lemma}
\newtheorem{proposition}[theorem]{Proposition}

\theoremstyle{definition}
\newtheorem{definition}[theorem]{Definition}
\newtheorem{remark}[theorem]{Remark}

\newtheorem{example}[theorem]{Example}
\newtheorem{question}[theorem]{Question}

\DeclareMathOperator{\Ima}{im}

\newcommand{\Aut}{{\rm Aut}}

\newcommand{\HK}{{\rm HK}}
\newcommand{\Homeo}{{\rm Homeo}}
\newcommand{\id}{{\rm id}}

 \newcommand{\NRP}{{\bf NRP}}
 \newcommand{\NRPk}{{\bf NRP}^{[k]}}
 \newcommand{\NRPs}{{\bf NRP}^{[s]}}

\newcommand{\Stab}{{\rm Stab}}

  \newcommand{\cD}{{\mathcal D}}

 \newcommand{\bN}{{\mathbb N}}

 \newcommand{\bR}{{\mathbb R}}
 \newcommand{\bT}{{\mathbb T}}
 \newcommand{\bZ}{{\mathbb Z}}
 
 \newcommand{\M}{ M}

\begin{document}
  \title{On the structure theory of cubespace fibrations}
  
  \author{Yonatan Gutman}
\author{Bingbing Liang}

\address{\hskip-\parindent
Y.Gutman, The Institute of Mathematics of the Polish Academy of Sciences,
ul. \'{S}niadeckich  8, 00-656  Warsaw, Poland}
\email{Y.Gutman@impan.pl}

\address{\hskip-\parindent
B.Liang, The Institute of Mathematics of the Polish Academy of Sciences,
ul. \'{S}niadeckich  8, 00-656  Warsaw, Poland }

\email{b.liang2018@outlook.com}
\email{bliang@impan.pl}


\maketitle

\begin{abstract}
We study fibrations in the category of cubespaces/nilspaces. We show that a fibration of finite degree $f \colon X\rightarrow Y$ between compact ergodic gluing cubespaces (in particular nilspaces) factors as a (possibly countable) tower of compact abelian Lie group principal fiber bundles over $Y$.  If the structure groups
of $f$ are connected then the fibers are (uniformly) isomorphic
(in a strong sense) to an inverse limit of nilmanifolds. In addition we give conditions under which 
the fibers of $f$ are isomorphic as subcubespaces.

We introduce regionally proximal 
equivalence relations relative to factor maps between minimal topological dynamical systems for an arbitrary acting group. We prove that
any factor map between minimal distal systems is a fibration and conclude that
if such a map is of finite degree then it factors as a (possibly
countable) tower of
principal abelian Lie compact group extensions, thus achieving  a refinement of both the Furstenberg's and the Bronstein-Ellis  structure theorems in this setting.
\end{abstract}

\tableofcontents

\section{Introduction}

\subsection{General background}

The theory of \emph{cubespaces} and its important subclass of \emph{nilspaces} originated in the work of Host and Kra \cite{HK08} under the name of \emph{parallelepiped structures}, and was developed further by Antolín Camarena and Szegedy \cite{ACS12}. A \emph{cubespace} is a structure consisting of a compact metric space $X$, together with a closed collection of \emph{cubes} $C^k(X)\subseteq X^{2^k}$ for each integer $k\ge 0$, satisfying certain natural axioms. A \emph{nilspace} is a cubespace satisfying an additional rigidity condition \footnote{For the exact definitions of the class of \emph{cubespaces} and the subclass of \emph{nilspaces} see Definitions \ref{cubespace} and \ref{nilspace} respectively.}. The theory has already found applications in \emph{higher order Fourier analysis} \cite{szegedy2012higher, tao2012higher}, in particular in relation to the inverse theorem for the Gowers norms \cite{GTZ12}, as well as in ergodic theory \cite{gutman2019strictly,candela2018nilspace} in relation to the Host-Kra structure theorem \cite{HK05,host2018nilpotent}. Cubespaces and nilspaces also played an essential role in the structure theory of the higher order \emph{nilpotent regionally proximal relations} 
introduced and developed by Glasner, Gutman and Ye in \cite{GGY18}.

As the theory of cubespaces/nilspaces has matured, it has been observed that \emph{fibrations}, cubespace morphisms satisfying an additional rigidity condition, are highly useful\footnote{For the exact definition of the class of \emph{fibrations} (of \emph{finite degree}) see Definition \ref{fibration definition} (Definition \ref{def:fibration_finite_degree}).}. 
This notion was introduced in \cite{GMVI} generalizing the previous notion of \emph{fiber-surjective morphism} from  \cite{ACS12}. Indeed
in \cite{GMVI, GMVII, GMVIII} Gutman, Manners and Varj\'{u} developed a weak structure theory for fibrations as an important step in the proof of the structure theorem for minimal topological dynamical \emph{systems of finite degree}, i.e., such that the nilpotent regionally proximal relation of some degree is trivial. According to this theorem such a system may be represented as an \emph{inverse limit of nilsystems} (subject to some mild assumptions on the acting group). 

According to the \emph{weak structure theorem} a fibration of finite degree factors
as a finite tower of compact abelian group extensions. The groups appearing in this factorization are referred to as the \emph{structure groups} of the fibration. 

In this paper we give a finer structure theory for fibrations building on the Antolín Camarena-Szegedy fundamental structure theorems for nilspaces. We show that a fibration of finite degree between cubespaces obeying some natural conditions factors as a (possibly countable) tower of \emph{Lie-fibered fibrations}, i.e., fibrations whose structure groups are 
(compact abelian) Lie groups. 

Given a fibration of finite degree $f \colon X \to Y$  it is well known that the \emph{fibers} $f^{-1}(y)$ are nilspaces. However this by itself does not elucidate the relation between \emph{different} fibers. In this paper we are able to give conditions guaranteeing that all fibers are \emph{isomorphic as cubespaces}. Moreover, taking advantage of the above-mentioned factorization into a tower of Lie-fibered fibrations, if the structure groups of the fibers are connected\footnote{In Proposition \ref{prop:structure groups} we show that the structure groups of \emph{all} fibers are connected iff the structure groups of $f$ are connected.} we show that the fibers are approximated uniformly as closely as desired  by quotients by natural co-compact subgroups of a \emph{single} Lie group associated with the factorization. 
 
We relate our results to the theory of topological dynamical systems. We show that any factor map between minimal distal systems $\pi\colon (G,X) \to (G,Y)$ where $G$ is an arbitrary topological group, is a fibration between the associated dynamical cubespaces. This supplies an abundance of hitherto unknown new examples of fibrations.   

The \emph{regionally proximal relations} have a long history in topological dynamics \cite{EG60,veech1968equicontinuous,ellis1971characterization}. Its relativization, the \emph{relativized regionally proximal relations} play a fundamental role in the study of the structure of equicontinuous extensions \cite{MW73, McM78}. Attesting to its importance is its central use in Bronstein's proof \cite{Bro68, bronshtein1970distal}\footnote{See also \cite[Chapter 7]{A}.} of the (relative) Furstenberg structure theorem for minimal distal extensions \cite{furstenberg1963structure}\footnote{The relative Furstenberg structure theorem was also proven independently by Ellis in \cite{E68}.}.

In this article we introduce the \emph{relative nilpotent regionally proximal relations} for extensions between minimal systems and analyze its structure. 
In particular we show that when an extension between minimal distal systems has trivial relative nilpotent regionally proximal relation of some degree, then it factors as a (possibly countable) tower of  \emph{principal abelian Lie (compact) group extensions}\footnote{For the exact definitions of \emph{ principal abelian group extensions} see Definition \ref{def:principal abelian group}.}.


\subsection{Cubespaces and nilspaces}

In this subsection, we define the notions of cubespace and nilspace and survey their important properties. For more detailed information see \cite{ACS12, candela2016cpt_notes, candela2016alg_notes, candela2019nilspace_morphisms, GGY18, GMVI,GMVII, GMVIII}.
\begin{definition} \label{cubespace}
Let $k, \ell \geq 0$ be two integers. A map $f=(f_{1},\ldots,f_{\ell}):\{0,1\}^{k}\to\{0,1\}^{\ell}$ is called
a \textbf{morphism of discrete cubes} if each coordinate function
$f_{j}(\omega_{1},\ldots,\omega_{k})$ is either identically  $0$,
identically $1$, or equals either $\omega_{i}$
or $\overline{\omega_{i}}=1-\omega_{i}$ for some $1\le i=i(j)\le k$.
If $\ell \leq k$, by an {\bf $\ell$-face} of $\{0, 1\}^k$ we mean a subset of $\{0, 1\}^k$ obtained from fixing values of  $(k-\ell)$ coordinates. In particular, a $(k-1)$-face is called a {\bf hyperface}\index{hyperface}.

A {\bf cubespace} \index{cubespace} is a metric space $X$, associated with a sequence of closed subsets $C^\bullet(X):=\{\index[nota]{C^k(X)} C^k(X) \subseteq  X^{\{0, 1\}^k} : k=0, 1, \ldots \}$  satisfying: 
\begin{enumerate}
    \item $C^0(X)=X$;
    \item for every integer $k, \ell \geq 0$, $c \in C^\ell(X)$, and morphism of discrete cubes $\rho: \{0, 1\}^k \to \{0, 1\}^\ell$, one has that 
    $c \circ \rho \in C^k(X)$.
\end{enumerate}
We call the elements of $C^k(X)$ $k$-{\bf cubes}\index{$k$-cube}.  A map $\{0, 1\}^k \to X$ is called a $k$-{\bf configuration}\index{configuration}. 
\end{definition}

A particular class of cubespaces arise from dynamical systems. 
By a {\bf (topological) dynamical system}\index{dynamical system} $(G, X)$\index{$(G, X)$},  we mean a continuous action of a topological group $G$ on a compact metric space $X$. 
\begin{definition} \label{HK cube}
Let $(G, X)$ be a dynamical system. For every integer $k \geq 0$ the {\bf Host-Kra cube group $\HK^k(G)$} \index{Host-Kra cube group}  is the subgroup of $G^{\{0, 1\}^k}$ generated by $[g]_F$ for all $g \in G$ and hyperfaces $F$ of $\{0, 1\}^k$. Here the configuration $[g]_F\index{$[g]_F$} \in G^{\{0, 1\}^k}$ sends $\omega \in F$ to $g$ and $e_G$ elsewhere. The cubespace $C_G^\bullet(X)$ is defined by taking orbit closures of constant configurations, i.e.
$$C_G^k(X):=\overline{\{\gamma . x^{\{0, 1\}^k}: \gamma \in \HK^k(G), x \in X \}},$$
where "." is denotes the pointwise action of $\HK^k(G)\subset G^{\{0, 1\}^k}$ on $X^{\{0, 1\}^k}$.  
We  call $(X, C_G^\bullet(X))$ a {\bf dynamical cubespace}\index{cubespace!dynamical cubespace}. Similarly if  $(X, C_G^\bullet(X))$ is a nilspace then we call it  a {\bf dynamical nilspace}.\index{nilspace !dynamical nilspace}
\end{definition}

The cubespaces category has natural notions of {\bf subcubespaces}, {\bf quotients} of cubespaces, and {\bf inverse limits} of cubespaces. \index{cubespace!subcubespace} \cite[Remark 3.7, Definition 5.2]{GMVI}. Let us review inverse limits in the category of cubespaces. Let $X :=\varprojlim X_i$ be an inverse system of cubespaces $(X_i, C^\bullet(X_i) )$. Then the cubespace structure on $X$ is defined via $$C^k(X):=\varprojlim C^k(X_i)$$
where the projections $ C^k(X_{i+1}) \to C^k(X_i)$ are induced from the pointwise projections  $\{X_{i+1} \to X_i\}_i$.

A cubespace $X$ is called {\bf ergodic} \index{ergodic} if every pair of points in $X$ is a $1$-cube, i.e. $C^1(X)=X^{\{0, 1\}}$. 
Recall that a dynamical system $(G, X)$ is called {\bf minimal} if the orbit of every point of $X$ is dense in $X$. A simple observation is that if $(G, X)$ is minimal then the induced  dynamical cubespace $(X, C^\bullet_G(X))$ is ergodic.

A cubespace $X$ is called {\bf strongly connected}\index{strongly connected} when all $C^k(X)$ are connected.

The first important property of cubespace is the (corner) completion property.
Denote by $\overrightarrow{1}$ the  element $(1, 1, \ldots, 1) \in \{0, 1\}^k$ and  $\llcorner^k$\index{$\llcorner^k$} the set $\{0, 1\}^k\setminus \{ \overrightarrow{1} \}$. Let $X$ be a cubespace, $k \geq 0$,  and $\lambda \colon \llcorner^k \to X$ a map. We call $\lambda$ a {\bf $k$-corner} \index{$k$-corner}
if every lower face of $\lambda$ is a $(k-1)$-cube, i.e. $\lambda|_{\{\omega: \ \omega_i=0\}} \in C^{k-1}(X)$ for all $i=1,\ldots, k$.
\begin{definition}
We say $X$ has {\bf $k$-completion} \index{$k$-completion} if every $k$-corner $\lambda$ of $X$ can be completed to a $k$-cube of $X$, i.e. $\lambda=c|_{\llcorner^k}$ for some $c \in C^k(X)$. A cubespace $X$ is called {\bf fibrant} \index{fibrant} if it has $k$-completion for every $k \geq 0$ (note that a $0$-corner is the empty set).
\end{definition}
Let $d$ be the mertic on a compact metric space $X$. Recall that a dynamical system $(G, X)$ is called {\bf distal} if $\inf_{g \in G} d(gx, gx') >0$ for any distinct points $x, x' \in X$.
\begin{example}
Let $(G, X)$ be a minimal distal system. Then the associated Host-Kra cubespace is fibrant \cite[Theorem 7.10]{GGY18}. In general, it is not the case for nondistal systems \cite[Example 3.10]{TY13} \cite[Example 9.3]{GGY18}. 
\end{example}
An important property of fibrant cubespaces is that they are gluing \cite[Proposition 6.2]{GMVI}. Let us recall the definition  as follows. 
Let $c_1, c_2\colon \{0, 1\}^k \to X$ be two configurations. The {\bf concatenation} \index{concatenation} $ [c_1, c_2]\index{$[c_1, c_2]$}\colon \{0, 1\}^{k+1} \to X$ of $c_1$ and $c_2$ is defined by
sending $(\omega, 0)$ to $c_1(\omega)$ and $(\omega, 1)$ to $c_2(\omega)$ for every $\omega \in \{0, 1\}^k$.

\begin{definition}
We say a cubespace $X$ has the  {\bf gluing property}  or is  {\bf gluing} \index{gluing}
if for every integer $k \geq 0$ and every $c_1, c_2, c_3 \in C^k(X)$, $[c_1, c_2], [c_2, c_3] \in C^{k+1}(X)$ implies that 
$[c_1, c_3] \in C^{k+1}(X)$.
\end{definition}

Another important property of cubespaces is the uniqueness property.
\begin{definition}\label{nilspace}
 A cubespace $X$ has {\bf $k$-uniqueness} \index{$k$-uniqueness} if for any $c_1, c_2 \in C^k(X)$ such that
$c_1|_{\llcorner^k}=c_2|_{\llcorner^k}$, one has $c_1=c_2$. Fix $s \geq 0$. We say  $X$  is a {\bf nilspace of degree at most s} or simply an  {\bf $s$-nilspace} \index{$s$-nilspace}
if it is fibrant and has (s+1)-uniqueness. We say $X$ is a {\bf nilspace} \index{nilspace} if it is an $s$-nilspace for some integer $s$.
\end{definition}

Note that if $X$ has $k$-uniqueness then $X$ has $\ell$-uniqueness for every $\ell \geq k$ as one  can apply the $k$-uniqueness to some suitable $k$-face of a given $\ell$-cube.

 \begin{definition} \label{canonical equivalence}
For a cubespace $X$, we say a pair of points $(x, x')$ of $X$ are {\bf $k$-canonically related}, denoted by $x \sim_k x'$\index{$\sim_k$} if there exists $c, c' \in C^{k+1}(X)$ such that $c|_{\llcorner^{k+1}}=c'|_{\llcorner^{k+1}}$, $c(\overrightarrow{1})=x$, and $c'(\overrightarrow{1})=x'$. The relation $\sim_k$  is called the {\bf $k$-th canonical  relation}\index{$k$-th canonical relation}.
\end{definition}

For compact gluing cubespaces the $k$-th canonical relation is an equivalence relation (\cite[Proposition 6.3]{GMVI}). This follows as compact gluing cubespaces satisfy the so-called {\bf universal replacement property}\index{universal replacement property} \cite[Proposition 6.3]{GMVI} (see also \cite[Lemma 2.5]{ACS12} and \cite[Proposition 3]{HK08}):
\begin{proposition} \label{URP}
Let $X$ be a compact gluing cubespace. Fix $s \geq 0$. Let $k \leq s+1$ and $c \in C^k(X)$. Then if a configuration $c' \in  X^{\{0, 1\}^k}$ has the same image as $c$ under the quotient map $X \to X/\sim_s$, then $c' \in C^k(X)$. 
\end{proposition}

 \begin{proposition}\cite[Proposition 6.2]{GMVI} \label{fibrant is gluing}
 Fibrant cubespaces (in particular nilspaces) satisfy the gluing property.
 \end{proposition}

Combining Propositions \ref{URP} and \ref{fibrant is gluing}, we have that the universal replacement property for a fibrant cubespace implies $(s+1)$-uniqueness for
the quotient cubespace $X/\sim_s$. Indeed we have:
\begin{corollary} \label{inducing nilspace}
 Let $X$ be a compact fibrant cubespace. Then $X/\sim_s$ is an $s$-nilspace for every $s \geq 0$.
\end{corollary}

The {\bf weak structure theorem}\index{weak structure theorem} for nilspaces of finite degree is an important factorization result into a finite tower of compact abelian group principal fiber bundles. We first recall the definition of a principal fiber bundle and then state the theorem.

\begin{definition} \cite[Definition 2.2]{H94}\label{def:principal fiber bundle}
Let $G$ be a topological group. A $G$-\textbf{principal fiber bundle} is a surjective continuous map $p \colon E \to B$ between two topological spaces $E$ and $B$ satisfying the following:
\begin{enumerate}
\item there exists a free continuous action of $G$ on $E$ such that for every $x \in E$
$$p^{-1}(p(x))=Gx;$$
\item there exists a homeomorphism $\varphi: B \to E/G$ such that $\varphi \circ p$ is the projection map $E \to E/G$.
\end{enumerate}
\end{definition}

\begin{theorem}  \cite[Theorem 5.4]{GMVI}\cite[Theorem 1]{ACS12} \label{WST}
Let $X$ be a compact ergodic $s$-nilspace. Then $X$ factors as
$$X=X/\sim_s \to X/\sim_{s-1} \to \cdots \to X/\sim_0 \cong \{\ast\},$$
where $\{\ast\}$ is a singleton and each map $X/\sim_k \to X/\sim_{k-1}$ is an $A_k$-principal fiber bundle for some compact metrizable abelian group $A_k$.
\end{theorem}
\noindent The group  $A_k$ is called the {\bf $k$-th structure group} of $X$. When all $A_k$ are Lie groups, $X$ is called {\bf Lie-fibered}\index{Lie-fibered}. 

\subsection{Host-Kra cubespaces}

Given a topological group $G$, a sequence of decreasing closed subgroups 
$$G=G_0 \supseteq G_1 \supseteq \cdots \supseteq G_{s+1}=\{e_G\}=G_{s+2}=\cdots$$
is an {\bf s-filtration} or a {\bf filtration of degree s} if $[G_i, G_j] \subseteq G_{i+j}$ for all $i, j \geq 0$ (here $[\cdot, \cdot]$ denotes the commutator subgroup). In such a case, we call $G$ a {\bf filtered group}\index{filtered group}  and
write $G_\bullet$\index{$G_\bullet$} to emphasize that it is equipped with some  s-filtration.
\begin{definition} \label{Host-Kra cube}
For a filtered metric group $G_\bullet$, the {\bf Host-Kra $k$-cube group} $\HK^k(G_\bullet)$\index{$\HK^k(G_\bullet)$} is defined 
as the subgroup of $G^{\{0, 1\}^k}$ generated by $[x]_F$ for every face $F \subseteq \{0, 1\}^k $ and $x \in G_{(k-\dim (F))}$.
\end{definition}

We remark that the Host-Kra $k$-cube group $\HK^k(G)$ in Definition \ref{HK cube} can be recovered  as $\HK^k(G_\bullet)$ with respect to the 1-filtration:
$$G=G_0=G_1 \supseteq G_2=\{e_G\}=\cdots.$$

When $G$ is a filtered topological group and $\Gamma$ is a discrete cocompact subgroup of $G$, 
 $G/\Gamma$ carries a quotient cubespace structure inherited from the Host-Kra cube group
$\HK^k(G_\bullet)$, which we denote by $\HK^k(G_\bullet)/\Gamma$ (see \cite[Definition 2.4]{GMVI} for more details).

We summarize the properties of  \emph{Host-Kra cubes} as follows \cite[Appendix A.4, Propositions 2.5, 2.6]{GMVI}.
\begin{proposition}
$(G, \HK^\bullet(G_\bullet))$ is a nilspace. Moreover, suppose that $\Gamma$ is {\bf  compatible with $G_\bullet$} in the sense that $\Gamma \cap G_i$ is discrete and cocompact
in $G_i$ for all $i \geq 0$. Then for each $k \geq 0$, $\HK^k(G_\bullet)/\Gamma \subseteq (G/\Gamma)^{\{0, 1\}^k}$ is a compact subset. Hence
$(G/\Gamma, \HK^\bullet(G_\bullet)/\Gamma) $ is a compact nilspace.
\end{proposition}

 Given a Lie group $G$ and a discrete cocompact subgroup the quotient $G/\Gamma$  (which carries the structure of a manifold) is termed a {\bf nilmanifold}\index{nilmanifold}. Using nilmanifolds the structure of a nilspace can be described as follows.
\begin{theorem} \cite[Theorem 1.28]{GMVIII} \label{absolute char}
Suppose that $X$ is a compact ergodic strongly connected nilsapce. Then $X$ is isomorphic as a cubespace to an inverse limit $\varprojlim X_n$ of nilmanifolds $X_n$ endowed with Host-Kra cubes.
\end{theorem}

\subsection{Structure theorems for fibrations}\label{subsec:main theorems}

\begin{definition}
Suppose that $\varphi \colon X \to Y$ is a continuous map between two cubespaces $X$ and $Y$. We say
$\varphi$ is a {\bf cubespace morphism}\index{cubespace!cubespace morphism} if $\varphi$ sends every cube of $X$ to a cube of $Y$. That is,
the set $\{\varphi \circ c: c \in C^k(X)\}$ is contained in $C^k(Y)$
for every integer $k \geq 0$.  We say $\varphi$ is a {\bf cubespace isomorphism} or simply an {\bf isomorphism} if $\varphi$ is a bijection and both $\varphi$ and $\varphi^{-1}$ are cubespace morphisms.
\end{definition}

In the sequel, given an integer $n \geq 1$ and a map $\varphi: X \to Y$ between metric spaces,  we will use  the same notation $\varphi$ to denote the induced map $X^n \to Y^n$ by pointwise application of $\varphi$, when no confusion arises.

\begin{definition} \label{relative ergodic}
A  cubespace morphism $\varphi \colon X \to Y$ is called {\bf relatively $k$-ergodic} if for any $c \in C^k(Y)$ any configuration of $\varphi^{-1}(c)$ is a cube of $X$. 
\end{definition}
It is clear that if $\varphi$ is relatively $k$-ergodic, $\varphi$ is relatively $\ell$-ergodic for each $\ell \leq k$.

Relativizing the concept of corner-completion, fibrations are introduced in \cite[Definition 7.1]{GMVI}:
\begin{definition} \label{fibration definition}
A cubespace morhpism $f\colon X \to Y$ is called a {\bf fibration}\index{fibration} if $f$ has {\bf $k$-completion}\index{fibration!$k$-completion} for all $k \geq 0$.  That is, 
given a $k$-corner $\lambda$ in $X$, if  $f( \lambda)$ can be completed to a cube $c$ in $Y$, then $\lambda$ can be completed to a cube $c_0$ of $X$ such that $f(c_0)=c$. 
\end{definition}

\begin{example}
In the setting of Definition \ref{Host-Kra cube}, the quotient map $G \to G/\Gamma$ induces a fibration $\HK^k(G_\bullet) \rightarrow \HK^k(G_\bullet)/\Gamma$ \cite[Proposition A.17]{GMVI}.
\end{example}

It is clear that a composition of fibrations is a fibration. By a "corner-lifting" argument, we have the so-called
{\bf universal property}\index{fibration!universal property} of fibrations \cite[Lemma 7.8]{GMVI}:
\begin{proposition} \label{universal property}
Let $f: X \to Y$ be a cubespace morphism and $g: Y \to Z$ a map between cubespaces. Then if $f$ and $g\circ f$ are fibrations, so is $g$.
\end{proposition}

We recall the relative notion of uniqueness for a cubespace morphism.
\begin{definition}\label{def:fibration_finite_degree}
Let $f\colon X \to Y$ be a cubespace morphism. We say $f$ has {\bf $k$-uniqueness}\index{fibration!$k$-uniqueness} if for any $k$-cubes $c, c'$ of $X$ such 
that $c|_{\llcorner^k}=c'|_{\llcorner^k}$ and $f( c) =f( c')$ we have $c=c'$. Moreover, we call $f$ is {\bf a fibration of degree at most $s$} or simply an {\bf $s$-fibration}\index{fibration!$s$-fibration}
 if $f$  is a fibration and has $(s+1)$-uniqueness.
\end{definition}

\begin{definition}
Given a fibration $f\colon X \to Y$, two points $x, x' \in X$ are called {\bf $k$-canonically related relative to $f$}, denoted by $x\sim_{f, k} x'$ \index{$\sim_k$!$\sim_{f, k}$},  if $x\sim_k x'$ and $f(x)=f(x')$.
\end{definition}

The following is a relative version of Corollary \ref{inducing nilspace}.
\begin{proposition}\cite[Proposition 7.12]{GMVI} \label{inducing $s$-fibration}
Let $f\colon  X \to Y$ be a fibration between compact gluing cubespaces. Fix $s \geq 0$. Then the relation $\sim_{f, s}$ is a closed equivalence relation and the projection map $\pi_{f, s}\index{$\pi_{f, s}$} \colon X \to X/\sim_{f, s}$ is a fibration and factors through an $s$-fibration $g\colon X/\sim_{f, s} \to Y$, that is,  the relation induces a commutative diagram
$$\xymatrix{
X \ar[dd]^f \ar@{-->}[dr]^{\pi_{f, s}} & \\
& X/\sim_{f, s} \ar@{-->}[dl]^g \\
Y.
}$$
\end{proposition}
We remark that  the  dashed arrows in the above proposition emphasize that the underling maps are induced from the given map. This convention will be used throughout the paper.

The induced $s$-fibration $g: X/\sim_{f, s} \to Y$ is maximal in the following sense. 
\begin{proposition} \label{maximal fibration}
Suppose that $f\colon X \to Y$ is a fibration and $s \geq 0$. Then the induced fibration $g\colon X/\sim_{f,s} \to Y$ is the {\bf maximal $s$-fibration} in the sense that for each commutative diagram of cubespace morphisms
$$\xymatrix{
X \ar[r]^{\pi'} \ar[d]_f & Z \ar[1, -1]^{g'} \\
Y &
}$$
such that $\pi'$ is a fibration and $g'$ is an $s$-fibration. Then $\pi\colon X \to X/\sim_{f, s}$ factors through $\pi'$, i.e. there is a  unique fibration $\varphi: X/\sim_{f,s} \to Z$  for which the following diagram is commutative:
$$\xymatrix{
X \ar[r]^{\pi} \ar[d]_{\pi'} & X/\sim_{f, s} \ar@{-->}[1, -1]^\varphi \\
Z. &
}$$
\end{proposition}

Now we are ready to state the {\bf relative weak structure theorem}\index{relative weak structure theorem} for fibrations \cite[Theorem 7.19, Corollary 7.20]{GMVI}. 
\begin{theorem} \label{RWST}
Let $f\colon X \to Y$ be an $s$-fibration between compact ergodic gluing cubespaces. Then $f$  factors as a finite tower of fibrations
$$\xymatrix{
X =X/\sim_{f, s} \ar[r] \ar[d]^f & X/\sim_{f, s-1} \ar[r] & \cdots \ar[r] & X/\sim_{f, 1} \ar[dlll]  \\
Y \cong X/\sim_{f, 0}
},$$
where for each $s \geq k \geq 1$  the fibration $X/\sim_{f, k} \to X/\sim_{f, k-1}$ is an $A_k(f)$-principal fiber bundle for a compact metrizable abelian group $A_k(f)$.
\end{theorem}

The group  $A_k(f)$ in the above theorem is named the {\bf $k$-th structure group} of $f$. In particular, $A_s(f)$ is called the {\bf top structure group}. If all $A_k(f)$ are Lie groups, $f$ is called {\bf Lie-fibered} as well as a {\bf Lie fibration}\index{Lie fibration}. 

Observe that a cubespace $X$ is an $s$-nilspace is equivalent to saying that the map $X \to \{\ast\}$ is an $s$-fibration.  The following proposition elaborates on this point. 
\begin{proposition} \label{fiber of fibration}
Let $f\colon X \to Y $ be an $s$-fibration. Then each fiber of $f$, as a subcubespace of $X$, is an $s$-nilspace.
\end{proposition}

The main goal of this paper is to describe the structure of an $s$-fibration.
Based on the relative weak structure theorem, our main endeavor is  to relativize various techniques which appeared in the absolute setting, i.e., we first consider the Lie-fibered case and then deal with the general case.

 One of the innovations in this article is to associate to a fibration $f\colon X \to Y$ the  {\it $k$-th translation group $\Aut_k(f)$ of $f$} (Definition \ref{relative translation}). Building on this concept, the  main result we obtain is as follows.

\begin{theorem} \label{relative inverse limit}
Let $g\colon Z \to Y$ be a fibration of degree at most $s$ between compact ergodic gluing cubespaces. Then there exists a sequence of compact ergodic gluing cubespaces $\{Z_n\}_{n\geq 0}$, and an inverse system of fibrations  $\{p_{m, n}\colon Z_n \to Z_m\}_{n \geq m}$ of degree at most $s$ satisfying the following properties:
\begin{enumerate}
    \item $Z$ is isomorphic to the inverse limit $\varprojlim Z_n$;
    \item There exists a sequence of  Lie fibrations $\{h_n\colon Z_n \to Y\}$ that are compatible with the connecting maps $p_{m, n}$;
    \item  Suppose that $g^{-1}(g(z))$ is strongly connected for some $z \in Z$. Define $z_n=p_n(z)$ as the image of the projection map $p_n\colon Z \to Z_n$. Then $g^{-1}(g(z))$  is isomorphic to the inverse limit of nilmanifolds $$\varprojlim (\Aut_1^\circ(h_n)/\Stab(z_n)).$$  Moreover, for each $k \geq 0$, $C^k(g^{-1}(g(z)))$ is isomorphic to the inverse limit
$$\varprojlim (\HK^k(\Aut_\bullet^\circ(h_n))/\Stab(z_n)).$$
\end{enumerate}
\end{theorem}

We remark that statements $(1)$ and $(2)$ of the above theorem imply that an $s$-fibration factors as an inverse limit of Lie fibrations. This is analogous to the statement in the absolute setting which says that an $s$-nilspace equals an inverse limit of Lie-fibered $s$-nilspaces. 

By Proposition \ref{fiber of fibration} the fiber $g^{-1}(z)$ is a nilspace. As by assumption $g^{-1}(z)$ is strongly connected, it holds by Theorem \ref{absolute char}, that it is isomorphic as a cubespace to an inverse limit of nilmanifolds endowed with the Host-Kra cubes. However statement (3) is much stronger as it gives a
\emph{simultaneous} representation for \emph{all} fibers of $f$ in
terms of the fixed sequence of Lie groups  $\Aut_1^\circ(h_n)$.
Thus  the fibers are approximated uniformly as closely as desired
by quotients by natural co-compact subgroups of a single Lie group
associated with
the factorization\footnote{This follows from the following lemma:
For an inverse limit $Z=\varprojlim Z_n$ arising from continuous maps $p_n \colon (Z, d) \to (Z_n, d_n)$ between compact metric spaces, we have  $\lim_{n \to \infty} \sup_{y \in Z_n} {\rm diam} (p^{-1}(y), d)$=0. As each $Z_n$ is homeomorphic to a disjoint
union of quotients of  $\Aut_1^\circ(h_n)$, the approximation is
uniform for all fibers simultaneously. }. 
We note that the fact that if one fiber is strongly connected then all fibers are strongly connected follows from Proposition \ref{prop:structure groups} in the sequel.

It is natural to ask what can be said for non strongly connected fibers. The general answer is not known, however in \cite{R78} Rees exhibited a certain minimal distal extension of a rotation on a solenoid for which all fibers are \emph{not} homeomorphic. As such an extension is a fibration we see that necessarily its fibers are not isomorphic. Rees also proved that the fibers of a minimal distal extension of a path-connected system are homeomorphic. Inspired by this result we prove the following theorem, relying on a homotopy argument and the relative weak structure theorem.

\begin{theorem} \label{fiber isomorphism}
 Let $g\colon Z \to Y$ be a fibration of degree at most $s$ between compact ergodic gluing cubespaces. Suppose that $Y$ is path-connected and $g$ is Lie-fibered.  Then for any $y_0, y_1 \in Y$ there is a cubespace isomorphism between $g^{-1}(y_0)$ and 
$g^{-1}(y_1)$ as cubespaces.
\end{theorem}

\subsection{Fibrations arising from dynamical systems}\label{subsec:Dynamical fibrations}
Proving that a cubespace morphism is a fibration is in general non-trivial. Turning to topological dynamical systems we show they provide a hitherto unknown source of  new examples of fibrations.  
\begin{theorem} \label{dyn_factor_is_fibration}
Let $\pi\colon X \to Y$ be a factor map of minimal systems such that $X$ is fibrant (e.g., $X$ is minimal distal).  Then $\pi$ is a fibration between the associated dynamical cubespaces.
\end{theorem}

The proof given in Section \ref{sec:Factor maps between minimal distal systems are fibrations} uses the relative Furstenberg structure theorem for distal extensions of minimal systems. 
\begin{remark}
Another rich source of $s$-fibrations is found in ergodic theory. In \cite{gutman2019strictly}, Gutman and Lian studied the question when an ergodic abelian group extension of a strictly ergodic system admits a strictly ergodic distal model. Under some sufficient conditions they  proved that the associated topological model map is actually an $s$-fibration. 
\end{remark}

\subsection{The relative nilpotent regionally proximal relation}

In \cite{GGY18} Glasner, Gutman, and Ye  introduced the so-called \emph{nilpotent regionally proximal relations} for actions of arbitrary groups. For minimal actions by abelian groups these relations coincide with the \emph{regionally proximal relations} introduced by Host, Kra and Maass \cite{HKM10}, generalizing the classical (degree $1$) definition of Ellis and Gottschalk \cite{EG60}. The regionally proximal relations are not equivalence relations in general. However  \cite{veech1968equicontinuous,ellis1971characterization, McM78} proposed  various sufficient conditions under which the regionally proximal relations of degree $1$ are equivalence relations. In particular this is the case for minimal abelian group actions \cite{SY12}. In \cite{GGY18} it was proven unexpectedly that for \emph{any} minimal action $(G, X)$, the nilpotent regionally proximal relation of degree $s$, denoted by $\NRPs(X)$,  is an equivalence relation and moreover using \cite[Theorem 1.4]{GMVIII}, under
some restrictions on the acting group, it corresponds to the maximal dynamical nilspace factor (also known as \emph{pronilfactor}) of order at most $s$. Moreover $\NRP^{[1]}(X)$ corresponds to the maximal abelian group factor of $(G, X)$ for any acting group $G$.

We recall the definition from \cite{GGY18}. Let $x, x' \in X$ be two points of a metric space $X$. Denote by $\llcorner^k(x, x')$\index{$\llcorner^k(x, x')$} the map $\{0,1\}^k \to X$ assigning the value $x'$ at $\overrightarrow{1}$ and $x$ elsewhere. 

\begin{definition} \label{def:NRP}
Let $(G,X)$ be a dynamical system. We say a pair of points $(x, x')$ of $X$ are {\bf nilpotent regionally proximal of order $k$}\index{nilpotent regionally proximal of order $k$}, denoted by $(x, x') \in \NRPk(X)$\index{$\NRPk(X)$}, if $\llcorner^{k+1}(x, x') \in C^{k+1}_G(X)$. 
\end{definition}

\begin{theorem}\cite[Theorem 3.8]{GGY18}\label{equivalence relation}
If $(G, X)$ is minimal, then $\NRPk(X)$ is a closed $G$-invariant equivalence relation for every $k \geq 0$.
\end{theorem}

\begin{definition} \label{relative NRP}
Let $\pi \colon (G, X) \to (G, Y)$ be a factor map of dynamical systems. We define the {\bf relative nilpotent regionally proximal relation of order $k$ w.r.t. $\pi$}, denoted by $\NRPk(\pi)$\index{$\NRPk(X)$!$\NRPk(\pi)$},  as the intersection of $\NRPk(X)$ with $R_\pi\index{$R_\pi$}:=\{(x, x') \in X^2: \pi(x)=\pi(x')\}$.
\end{definition}

\begin{proposition} \cite[Lemma 7.17]{GMVI} \label{alternative}
For a factor map $\pi: X \to Y$ of dynamical systems, we have $\NRPk(\pi)=\sim_{\pi, k}$. In particular, $\NRPk(X)=\sim_k$.
\end{proposition}

From Theorem \ref{equivalence relation}, we see that for any factor map $\pi$ between minimal systems $\NRPk(\pi)$ is a closed $G$-invariant equivalence relation.

Let us discuss the relation between $\NRPk(\pi)$ and several classical (relative) relations.
\begin{definition}
Let $\pi: (G, X) \to (G, Y)$ be a factor map. The {\bf relative proximal relation} ${\bf P}(\pi)$ is defined as those pairs $(x, y) \in R_\pi$ such that
$d(g_ix, g_iy)$ approaches $0$ for some sequence $\{g_i\}_i$  of $G$. The map $\pi$ is called a {\bf distal extension} if ${\bf P}(\pi)$ is trivial. In particular, $(G, X)$ is called a  distal system if and only if ${\bf P}((G, X) \to \{\ast\})$ is trivial.  When $G$ is abelian, for every positive integer $k$, we define the {\bf relative $k$-th regionally proximal relation of $\pi$}, denoted by ${\bf RP}^{[k]}(\pi)$,  by the collection of pairs $(x, y) \in R_\pi$ such that there exist sequences of
elements $(x_i, y_i) \in R_\pi$ and $(g_i^1, \ldots, g_i^k) \in G^k$ satisfying that 
$$\lim_{i \to \infty} (x_i, y_i)=(x, y), {\rm \ and \ } \lim_{i \to \infty} d((\sum_{j=1}^k \epsilon_jg_i^j)x_i, (\sum_{j=1}^k \epsilon_jg_i^j)y_i)=0$$
for every $(\epsilon_1, \ldots, \epsilon_d) \in \{0, 1\}^d\setminus  \{\overrightarrow{0}\}$.
\end{definition}

Penazzi gave some algebraic conditions under which the relative regionally proximal relation is an equivalence relation \cite{P95}. More about this topic may be found in  \cite[Chapter V.2]{VriesB}.

Let $G$ be an abelian group and $\pi: (G, X) \to (G, Y)$ a factor map of minimal systems. From \cite[Proposition 8.9]{GGY18}, we have that ${\bf RP}^{[k]}(X) \subseteq \NRPk(X)$. As a consequence, we have the following proposition.
\begin{proposition} \label{relative relation}
Let $G$ be an abelian group and $\pi: (G, X) \to (G, Y)$ a factor map of minimal systems. Then 
 ${\bf P}(\pi) \subseteq {\bf RP}^{[1]}(\pi) \subseteq {\bf RP}^{[k]}(\pi) \subseteq \NRPk(\pi)$. In particular, If $\NRPk(\pi)$ is trivial, $\pi$ is a distal extension.
\end{proposition}

\begin{remark}
\begin{enumerate}
\item For a factor map $\pi \colon (G, X) \to (G, Y)$ of minimal systems, since ${\bf RP}^{[1]}(\pi) \subsetneq {\bf RP}^{[1]}(X)\cap R_\pi$ in general (even for $G=\bZ$) \cite[Remark (2.2.3), Page 411]{VriesB}, we have in particular in general that ${\bf RP}^{[1]}(\pi) \subsetneq {\NRP^{[1]}(\pi)}$. \\
\item  Recall that ${\bf RP}^{[1]}(\pi)$ generates the maximal equicontinuous factor $(G, X/{\bf RP}^{[1]}(\pi)) \to (G, Y)$  (see \cite[Chapter V, Theorem 2.21]{VriesB}). One may wonder whether $\NRP^{[1]}(\pi)$ corresponds to the maximal factor which is a principal abelian group extension\footnote{See Definition \ref{def:principal abelian group}.} of the base space. It turns out that this is not true. Indeed, from the famous Furstenberg counterexample \cite[Remark 7.19]{GGY18}, a principal abelian group extension of minimal (distal) systems is not necessarily of finite degree relative to the factor system. To see this, recall that Furstenberg's example is given by the projection map $\pi\colon (\bT^2, T) \to (\bT, S)$. Here $T$ sends $(x, y)$ to $(x+\alpha, y+\varphi(x))$ for some continuous map $\varphi: \bT \to \bT$ and irrational number $\alpha$. Since $(\bT, S)$ is a minimal abelian group system, by \cite[Lemma 8.4]{GGY18}, $\NRPk(\bT) \subseteq \NRP^{[1]}(\bT)$ are trivial for all $k \geq 1$. This implies that $\NRPk(\bT^2)\subseteq R_\pi$ and hence $\NRPk(\pi)=\NRPk(\bT^2)$ is not trivial.
\end{enumerate}
\end{remark}
\emph{Systems of finite degree}\footnote{Also known as \emph{systems of finite order}.} were introduced in \cite{HKM10} for $\mathbb{Z}$-actions. In \cite{GGY18} they were defined for general group actions and investigated from a structural point of view.  Recall that a minimal system $(G, X)$ is called a system of degree at most $s$ if $\NRPs(X)$ equals the diagonal subset $\Delta$ of $X^2$ (\cite[Definition 7.1]{GGY18}). It is thus natural to define:
\begin{definition}\label{def:extension of degree}
Let $\pi\colon X \to Y$ be a factor map of minimal systems. We say $\pi$ is an  {\bf extension of degree at most $s$ (relative to $Y$)} if $\NRPs(\pi)=\Delta$.
\end{definition}

In Section \ref{sec:extensions of finite degree} we develop the structure theory of extensions of finite degree. Our main structural theorem, proven in Section \ref{sec:extensions of finite degree}, is an application of Theorem \ref{relative inverse limit} in the dynamical setting:

\begin{theorem}\label{thm:structure_finite_degree_extension}
 Let $\pi \colon (G,X) \to (G,Y)$ be an extension of degree at most $s$ between minimal distal systems. Then the following holds:
 \begin{enumerate}
       \item The map $\pi$ is an $s$-fibration. 
     \item
     There exists a sequence of dynamical systems $\{(G,Z_n)\}_{n\geq 0}$, and an inverse system of factor maps $\{p_{m, n}\colon (G,Z_n) \to (G,Z_m)\}_{n \geq m}$ such that $(G,X)=\varprojlim (G,Z_n)$.
     \item
     There exists a sequence of factor maps  $\{h_n\colon (G,Z_n) \to (G,Y)\}$ which are Lie fibrations and are compatible with the connecting maps $p_{m, n}$.
 \end{enumerate}
 \end{theorem}

\begin{question}\label{q:fiber_isomorphic}
Let $\pi \colon X \to Y$ be an extension of degree at most $s$  between minimal systems such that  $\pi$ is a fibration. Is every fiber of $\pi$  isomorphic as a cubespace to an inverse limit of nilmanifolds?
\end{question}

\begin{remark}
It is important to point out that in the relative setting, $G$ does not immerse into $\Aut_1(\pi)$ because $G$ does not fix the fibers of $\pi$. This is a crucial obstruction not existing in the absolute setting, making Question \ref{q:fiber_isomorphic} non-trivial.
\end{remark}

\begin{remark}
It is interesting to compare Theorem \ref{thm:structure_finite_degree_extension} to a refinement of the relative Furstenberg structure theorem due to Bronstein (see \cite[3.17.8]{bronshteuin1979extensions}) which states that a distal extension of metric minimal dynamical systems factors as (a possibly countable) tower of \emph{isometric extensions}\footnote{See Definition \ref{isometric}.} where the fibers are given by quotients of compact Lie groups by closed subgroups. Note however that these compact Lie groups are not necessarily abelian. 
\end{remark}

\subsection{Conventions}
Throughout the paper, all spaces are metric spaces and $G$ denotes a topological group. When we discuss an $s$-fibration, the underlying cubespaces are always assumed to be compact ergodic and have the gluing property. Unless specified otherwise $k$ always denotes a non-negative integer.

\subsection{Structure of the paper}
Theorem \ref{relative inverse limit} is proven in Sections \ref{sec:Strongly connected fibers} and \ref{sec:Approximating by Lie-fibered fibrations}, where Section  \ref{sec:Strongly connected fibers} is solely devoted to the Lie-fibered case of statement (3). In Section \ref{sec:Isomorphisms between fibers of a fibration} we prove Theorem \ref{fiber isomorphism} and in Section \ref{sec:Factor maps between minimal distal systems are fibrations} we prove Theorem \ref{dyn_factor_is_fibration}. The final section, Section \ref{sec:extensions of finite degree} is devoted to
extensions of finite degree and to the proof of Theorem \ref{thm:structure_finite_degree_extension}.

\subsection{Acknowledgements}

Y.G. was partially supported by the National Science Centre (Poland) grant 2016/22/E/ST1/00448. B.L. would like to acknowledge the support from the Institute of Mathematics of the Polish Academy of Sciences.

\section{Strongly connected fibers}
\label{sec:Strongly connected fibers}
In this section, we prove the Lie-fibered case of Theorem \ref{relative inverse limit}(3). Recall that an $s$-fibration $g\colon Z \to Y$ between compact ergodic gluing cubespaces is called {\bf Lie fibered} if all structure groups $A_k(g)$ of $g$ are Lie groups.
In light of Theorem \ref{RWST}, we write $g_{s-1}: Z/\sim_{g, s-1} \to Y$\index{$g_{s-1}$} for the induced $(s-1)$-fibration, which we may call the {\bf canonical $(s-1)$-th factor of $g$}.
\subsection{The main steps of the proof}

Let us first recall the notion of the $k$-th translation groups of a cubespace. For a compact cubespace $X$, denote by $\Aut(X)$ the collection of cubespace isomorphisms of $X$.
\begin{definition} \label{translation}
Fix $k \geq 0$. We say $\varphi \in \Aut(X)$\index{$\Aut(X)$} is a {\bf $k$-translation} if for every $n \geq k$, $(n-k)$-face $F \subseteq \{0, 1\}^n$, and $c \in C^n(X)$, the map  $\{0, 1\}^n \to X$ sending $\omega \in F$  to $\varphi(c(\omega))$ and $c(\omega)$ elsewhere is still a cube of $X$.  The {\bf $k$-th translation group}\index{$k$-th translation group} $\Aut_k(X)$\index{$\Aut(X)$!$\Aut_k(X)$} of $X$ is defined as the collection of all $k$-translations of $X$.
\end{definition}

Let $d$ be the metric on $X$. For each homeomorphism $\phi \colon X \to X$, define
$$||\phi||:=\max_{x \in X} d(x, \phi(x)).$$
Then the compact-open topology of the group of homeomorphisms of $X$, denoted by ${\rm Homeo}(X)$, can be induced from the  metric
$$d(\phi, \psi):=||\phi \circ \psi^{-1}||.$$
In particular, as a closed subgroup of ${\rm Homeo}(X)$, $\Aut(X)$ is a metric space.
When we say $\varphi \in {\rm Homeo}(X)$ is (appropriately) {\bf small}, we mean $||\varphi||$ is (appropriately) small; if $f\circ \varphi=f$, we say $\varphi$ {\bf fixes the fibers of $f$}.

Now consider a fibration $f\colon X \to Y$. To study the cubespace structure of the fibers of $f$, we introduce the notion of the $k$-th translation group of $f$.
Note that each fiber $f^{-1}(y)$ is a subcubespace of $X$. 
 \begin{definition} \label{relative translation}
The {\bf $k$-th translation group of $f$}\index{$k$-th translation group!$k$-th translation group of $f$} is defined as 
$$\Aut_k(f)\index{$\Aut(X)$!$\Aut_k(f)$}:=\{\varphi \in \Homeo(X): \varphi|_{f^{-1}(y)} \in \Aut_k(f^{-1}(y)) \ {\rm for \ every \ } y \in Y \}.$$
 \end{definition}

Intuitively this means that each element $\varphi$ of $\Aut_k(f)$ is a homeomorphism of $X$ such that $\varphi$ fixes the fibers of $f$ and  preserves the cubespace structure of fibers in a strong sense. 

\begin{remark}
 It is interesting to compare this definition with the definition of the \emph{fiber preserving automorphism group}  of a dynamical extension $\phi\colon (G, X) \to (G, Y)$, used to characterize (weak) group extensions \cite[Chapter V, Remark 4.2 (4)]{VriesB}.
\end{remark}

When $f$ is an $s$-fibration, $\Aut_{s+1}(f)=\{1\}$ and  we obtain a filtration $\Aut_\bullet(f)$ of degree $s$ as follows:
$$\Aut_1(f)=\Aut_1(f) \supseteq \Aut_2(f) \supseteq \cdots \supseteq \Aut_{s+1}(f)=\{1\}.$$
We may therefore consider $\Aut_\bullet(f)$ as a cubespace (furthermore an $s$-nilspace) endowed with Host-Kra cubes.

Denote by $\Aut_k^\circ(f)$\index{$\Aut(X)$!$\Aut_k^\circ(f)$} the connected component of $\Aut_k(f)$ containing the identity. We obtain  a filtration of closed groups
$$\Aut_1^\circ(f)=\Aut_1^\circ(f) \supseteq \Aut_2^\circ(f) \supseteq \cdots \supseteq \Aut_{s+1}^\circ(f)=1.$$

Similarly to \cite[Proposition 3.1]{GMVII}, we have
\begin{proposition} \label{embedding}
Let $g\colon Z \to Y$ be an $s$-fibration between compact ergodic gluing cubespaces. Then $A_s(g)$ embeds into $\Aut_s(g)$ via sending $a $ to  $f_a$, where $f_a: Z \to Z$ is given by $f_a(z)=a.z$. 
\end{proposition}

\begin{proof}
Fix $n \geq s, y \in Y$, and $c \in C^n(g^{-1}(y))$. Define $A:=A_s(g)$ and $\pi:=\pi_{g, s-1}$. 
Since $C^n(\cD_s(A))c=\pi^{-1}(\pi(c))$ (see definition of $\cD_s(A)$ in \cite[Section 5.1]{GMVI}), we have 
$$g([a]_Fc)=g_{s-1}\circ \pi([a]_Fc)=g_{s-1}\circ \pi(c)=g(c)=y$$
for any face $F \subseteq \{0, 1\}^n$ of codimension $s$.
It follows that $f_a|_{g^{-1}(y)} \in \Aut_s(g^{-1}(y))$ and thus  $f_a \in \Aut_s(g)$.
\end{proof}

As a relative analogue of \cite[Proposition 2.17]{GMVII}, the following proposition says that we can also use the evaluation map to construct the desired cubespace morphism.
\begin{proposition} \label{morphism repr}
Let $g: Z \to Y$ be an $s$-fibration. Fix  $z \in Z$ and a cube $c$ of $C^n(g^{-1}(g(z)))$. Then for any $(\varphi_\omega)_{\omega \in \{0, 1\}^n}$ in $\HK^n(\Aut_\bullet(g))$,   $(\varphi_\omega(c(\omega)))_\omega \in C^n(g^{-1}(g(z)))$. In particular the induced natural map 
$$\Aut_1(g)/\Stab(z) \to g^{-1}(g(z))$$
is a cubespace morphism.
\end{proposition}

\begin{proof}
The definition of $\Aut_\bullet(g)$ guarantees that the map is well defined. The rest of the argument stays the same as in \cite[Proposition 2.17]{GMVII}.
\end{proof}

Recall that a cubespace $X$ is called {\bf strongly connected}\index{strongly connected} if $C^k(X)$ is connected for every $k \geq 0$. 
In \cite[Theorem 2.18]{GMVII}, a strongly connected Lie-fibered nilspace $X$ is shown to be a nilmanifold using the translation groups $\Aut_k(X)$.  For the relative case, we  are  now ready to state the structure theorem for a Lie-fibered fibration of finite degree. 
\begin{theorem}[The Lie-fibered case of Theorem \ref{relative inverse limit}(3)] \label{relative nilmanifold}
 Let $g\colon Z \to Y$ be a fibration of degree at most $s$ between compact ergodic cubespaces that obey the gluing condition. Fix a point $z \in Z$. Suppose that $g$ is Lie-fibered, then the following holds.
 \begin{enumerate}
     \item $\Aut_1^\circ(g)$ is a Lie group; \\
     \item the stablizer $\Stab(z)$ of $z$ in $\Aut_1^\circ(g)$ is a discrete cocompact subgroup;\\
     \item if $g^{-1}(g(z))$ is strongly connected as a subcubespace, then the fiber $g^{-1}(g(z))$ is homeomorphic to the nilmanifold $\Aut_1^\circ(g)/\Stab(z)$.  Moreover, ther are cubespace isomorphisms
 $$C^k(g^{-1}(g(z))) \cong \HK^k(\Aut_\bullet^\circ(g))/\Stab(z)$$
 for all $k \geq 1$.
 \end{enumerate}
\end{theorem}

The proof of Theorem \ref{relative nilmanifold} uses the full strength of the relative weak structure theorem (Theorem \ref{RWST}) and will be given at the end of this subsection. 
Recall that in \cite[Proposition 3.2]{GMVII}, for  an $s$-nispace $X$ a canonical group homomorphism  $\Aut_k(X) \to \Aut_k(X/\sim_s)$ is exhibited. We adapt the statement and argument for our relative setting. 
\begin{proposition} \label{canonical map}
Fix $s, k \geq 1$.  Let $g\colon Z \to Y$ be an $s$-fibration between compact ergodic gluing cubespaces. Then there is a canonical continuous group homomorphism $\pi_\ast\colon \Aut_k(g) \to \Aut_k(g_{s-1})$ such that for every $\varphi \in \Aut_k(g)$ the diagram 
$$\xymatrix{
Z \ar[r]^\varphi \ar[d]_{\pi_{g, s-1}} \ar@/^3pc/[drr]^g  & Z \ar[dr]^g \ar[d]_{\pi_{g, s-1}} & \\
\pi_{g, s-1}(Z) \ar@{-->}[r]^{\pi_\ast(\varphi)} \ar@/_2pc/[rr]^{g_{s-1}} & \pi_{g, s-1}(Z) \ar[r]^{g_{s-1}} & Y
}$$
commutes.
\end{proposition}

\begin{proof}
We modify the proof in the absolute setting to the relative case. Denote by $\pi$ the projection map $\pi_{g, s-1}$.  Given $\varphi \in \Aut_k(g)$ and $y \in \pi(Z)$, writing $y=\pi(x)$ for some $x \in Z$, we define $\pi_\ast(\varphi)(y):=\pi \circ \varphi(x)$.

We check that  $\pi_\ast$ is well-defined.  Note that $\varphi \in \Aut_k(g) \subseteq \Aut_1(g)$. By Proposition \ref{embedding}, $A_s(g)$ embeds into $\Aut_s(g)$.  Since $\Aut_\bullet(g)$ is filtered, we conclude  that $A_s(g)$ commutes with $\Aut_k(g)$.  For any $y \in \pi(Z)$ and $x, x' \in \pi^{-1}(y)$, there exists a unique $a \in A_s(g)$ such that $x'=ax$. Then
$$(\pi_\ast(\varphi))(y):=\pi \circ \varphi(x')=\pi\circ \varphi(ax)=\pi\circ (a\varphi(x))=\pi\circ \varphi(x).$$
This shows that $\pi_\ast$ is independent of the choice of lifting. The statement that $\pi_\ast(\varphi)$ is a $k$-translation follows from the facts that $\varphi$ is a $k$-translation and $\pi$ is a cubespace morphism.   
\end{proof}

The first hard ingredient of the proof of Theorem \ref{relative nilmanifold}, used for establishing the nilmanifold structure in statement (3), is to relativize  \cite[Proposition 3.3]{GMVII}: Every element of $\Aut_k(g_{s-1})$ of sufficient small norm can be realized as the image of $\pi_\ast$ of some element of $\Aut_k(g)$. The proof will be given in Section 4.2.
\begin{theorem}\label{openness}
Fix $s \geq 1$. Let $g: Z \to Y$ be an $s$-fibration between compact ergodic gluing cubespaces.  Assume  that $A_s(g)$ is  a Lie group. Then $\pi_\ast \colon \Aut_k(g) \to \Aut_k(g_{s-1})$ is open for all $k \geq 1$.
\end{theorem}

The next result is analogous to \cite[Theorem 5.2]{GMVII}. As the proof is similar we omit it.

\begin{theorem} \label{rigidity}
Let $A$ be a compact abelian Lie group with a metric $d_A$. Fix integers $s \geq 0$ and $\ell \geq 1$. Then there exists 
$\varepsilon$ depending only on $s, \ell$ and $A$ such that the following holds. 

Let $g: Z \to Y$ be an $s$-fibration between compact ergodic gluing cubespaces. Suppose that $f: Z \to A$ is a continuous function such that 
$$d_A(f(z), f(z')) \leq \varepsilon$$
for every $z, z' \in Z$ with $g(z)=g(z')$, and the map $\partial^\ell f: C^\ell_g(Z) \to A$ sending $c$ to 
$$\sum_{\omega \in \{0, 1\}^\ell} (-1)^{|\omega|} f(c(\omega))$$
 is the zero function. Then $f$ is a constant function.
\end{theorem}

The following is used in the the proof of Theorem \ref{relative nilmanifold}.
\begin{lemma} \label {discreteness}
Lie-fiberedness of $g$ implies that $\Stab(z) \subseteq \Aut_1(g)$ is discrete for every $z \in Z$.
\end{lemma}

\begin{proof}
In light of the relative weak structure theorem (Theorem \ref{RWST}), Theorem \ref{rigidity} guarantees that the argument in the absolute setting may be modified appropriately.
\end{proof}

The proof of the following lemma follows the  argument of \cite[Proposition 3.7]{GMVII} (with the help of  Lemma \ref{discreteness}) so we omit it.
\begin{lemma} \label{open subgroup}
$A_s(g) \subseteq \ker (\pi_*)$ is open.
\end{lemma}

Based on Lemma \ref{open subgroup}, we can prove one of the statements of Theorem \ref{relative nilmanifold}.
\begin{corollary} \label{Lieness}
 The $1$-translation group $\Aut_1(g)$ has a Lie group structure.
\end{corollary}

\begin{proof}
First note that $\Aut_1(g_0)=\Aut_1(\id_Y)=1$ and hence $\ker(\Aut_1(g_1) \to \Aut_1(g_0))=\Aut_1(g_1)$.  By Lemma \ref{open subgroup}, $A_1(g)$ is an open subgroup of $\ker(\Aut_1(g_1) \to \Aut_1(g_0))$. Thus we can extend the differentiable structure of $A_1(g)$ onto $\Aut_1(g_1)$.

Inductively  assume that $\Aut_1(g_{s-1})$ is Lie. By Theorem \ref{openness}, $\Ima (\pi_\ast)$ is an open subgroup of $\Aut_1(g_{s-1})$ and hence Lie. Applying the extension theorem for Lie groups \cite[Theorem 3.1]{Gle51}, it suffices to show $\ker (\pi_\ast)$ has a Lie group structure. By Lemma \ref{open subgroup}, $A_s(g) \subseteq \ker (\pi_\ast)$ is an open subgroup. Thus we can extend the Lie group structure of $A_s(g)$ onto $\ker (\pi_\ast)$.
\end{proof}

The following is an analogue of \cite[Lemma 3.10]{GMVII}.
\begin{lemma} \label{identity orbit is open}
Let $g\colon Z \to Y$ be a Lie-fibered $s$-fibration. Fix $0 \leq k < s$. Then for any $\varepsilon > 0$ there exists  $\delta > 0$ satisfying the following.
 For any $z, z' \in Z$ such that $z\sim_{g, k} z'$ and $d(z, z') < \delta$, there is  $\varphi \in \Aut_{k+1}(g)$ such that $||\varphi||< \varepsilon$ and $\varphi(z)=z'$. In particular, for each $z \in Z$, $\Aut_{k+1}^\circ(g)z$ is open in  $\pi_{g, k}^{-1}(\pi_{g, k}(z))$.
\end{lemma}

\begin{proof}
Based on Proposition \ref{canonical map} and Theorem \ref{openness}, similarly to the proof of \cite[Lemma 3.10]{GMVII}, we can induct on $s$ to obtain the first conclusion. We now prove that  $\Aut_{k+1}^\circ(g)z$ is open. Fix  $\varphi \in \Aut_{k+1}^\circ(g)$. Let $\varepsilon > 0$  be small enough such that the $\varepsilon$-ball $B_\varepsilon({\rm id})$ at the identity is contained in $\Aut_{k+1}^\circ(g)$.
There exists  $\delta >0$ satisfying that whenever $y, y' \in \pi_{g, k}^{-1}(\pi_{g, k}(z))$ such that $d(y, y')< \delta$, we can find $\varphi' \in B_\varepsilon({\rm id})$ such that
$y=\varphi'(y')$.

Note that for $y \in B_\delta(\varphi(z)) \cap \pi_{g, k}^{-1}(\pi_{g, k}(z))$, we have
$$\pi_{g, k}(y)=\pi_{g, k}(z)=\pi_{g, k}\circ \varphi(z).$$
It follows that $y =\varphi' \circ \varphi(z)$
for some $\varphi' \in \Aut_{k+1}^\circ(g)$. Q.E.D.
\end{proof}

The following is an analogue of \cite[Lemma 3.12]{GMVII}.
Using the relative weak structure theorem, the proof follows a similar argument, inducting on $s$.
\begin{lemma} \label{cubehom is open}
For each $c \in C^n(g^{-1}(g(z)))$, the evaluation map ${\rm ev}_c\colon \HK^n(\Aut_\bullet^\circ(g)) \to C^n(g^{-1}(g(z)))$ sending $(\varphi_\omega)_\omega$ to $(\varphi_\omega(c(\omega)))_\omega$ is open.
\end{lemma}

\begin{proof}[{\bf Proof of Theorem \ref{relative nilmanifold}}]
Since $\Aut_k^\circ(g) \subseteq \Aut_k(g) \subseteq \Aut_1(g)$ are closed subgroups, and by Corollary \ref{Lieness}, $\Aut_1(g)$ is Lie, we have $\Aut_k^\circ(g)$ is Lie.

To show that the evaluation map $\HK^n(\Aut_\bullet^\circ(g)) \to C^n(g^{-1}(g(z)))$ is surjective, it suffices to show that  the action $\HK^n(\Aut_\bullet^\circ(g)) \curvearrowright C^n(g^{-1}(g(z)))$ is transitive. Since $\Aut_k^\circ(g)$ is open in $\Aut_k(g)$ for all $k \geq 0$, by a relative version of \cite[Lemma 3.15]{GMVII}, $\HK^n(\Aut_\bullet^\circ(g))$ is open in $\HK^n(\Aut_\bullet(g))$. Thus by Lemma \ref{cubehom is open}, the orbits of $\HK^n(\Aut_\bullet^\circ(g))$ are open and hence closed in $C^n(g^{-1}(g(z)))$. 
As by assumption, $C^n(g^{-1}(g(z)))$ is connected, it implies that the action is transitive. In particular, the evaluation map ${\rm ev}_z\colon \Aut_1^\circ(g)/\Stab(z) \to Z$ and the induced maps 
$${\rm ev}_{\Box^n(z)}\colon \HK^n(\Aut_\bullet^\circ(g))/\Stab(z) \to C^n(g^{-1}(g(z)))$$
by pointwise application are homeomorphisms. Here $\Box^n(z)$\index{$\Box^n(z)$} denotes the constant cube of $C^n(Z)$ taking value $z$. 

To see $\Stab(z)\cap \Aut_k^\circ(g)$ is co-compact in $\Aut_k^\circ(g)$ for all $k \geq 0$, we first note by Lemma \ref{identity orbit is open} that $\Aut_k^\circ(g)z$ is open in $\pi_{g, k-1}^{-1}(\pi_{g, k-1}(z))$. 
Moreover, since orbits partition the space and  $\Aut_k^\circ(g)z$ is an orbit, it must be closed and hence compact. From the homeomorphism between $\Aut_k^\circ(g)z$ and $\Aut_k^\circ(g)/(\Stab(z)\cap \Aut_k^\circ(g))$, we see that $\Stab(z)\cap \Aut_k^\circ(g)$ is co-compact.

\end{proof}

\subsection{Lifting relative translations}

In this subsection, we prove Theorem \ref{openness}. There are two steps in the proof: (1) lift small translations $\varphi$ in $\Aut_k(g_{s-1})$  to  a small homeomorphism $\psi\colon Z \to Z$; (2)  replacing $\psi$ with a genuine $k$-translation of the form $h(\cdot).\psi(\cdot)$ for some continuous map $h\colon Z \to A_s(g)$ with small norm. We will refer to the operation of replacing   $\psi$ with a  $k$-translation of the form $h(\cdot).\psi(\cdot)$ as \textbf{repairing}. Step (2) will be based on a relative version of \cite[Theorem 4.11]{GMVII} \cite[Lemma 3.19]{ACS12}.

 By a careful local section argument for Lie-principal bundles as in \cite[Lemma 4.2]{GMVII}, we accomplish the first step by the following lemma.
 \begin{lemma} \label{preliftting}
Let $g\colon Z \to Y$ be an  $s$-fibration between compact ergodic gluing cubespaces. Assume that $A_s(g)$ is Lie. Then one can lift every small enough homeomorphism  $\varphi$ of $\pi_{g, s-1}(Z)$ up to a small homeomorphism $\psi \colon Z \to Z$ such that $\psi$ is $A_s(g)$-equivariant.
 \end{lemma}
 
 \begin{definition}
For a cubespace morphism $f\colon  X \to Y$,  we define the  {\bf $k$-th fiber cubes of $f$}\index{$k$-th fiber cube} as the closed subset 
$$C^k_f(X): =\cup_{y \in Y} C^k(f^{-1}(y)).$$
\end{definition}
 
 \begin{definition}
Let $\ell \geq 1$ and $n \geq 0$. Given $c_1, c_2 \in C^n(X)$, the {\bf generalized $\ell$-corner} $\llcorner^\ell(c_1; c_2)\colon \{0, 1\}^{n+\ell} \to X$  is given by $c_2(\omega_1, \ldots, \omega_n)$ for $\omega=(\omega_1, \ldots, \omega_n, \overrightarrow{1})$ and $c_1(\omega_1, \ldots, \omega_n)$ elsewhere.
 \end{definition}
 
We need to repair the lift $\psi$ to a $k$-translation of $\Aut_k(g)$. Unwrapping the definition, we want to find a small continuous function $h\colon Z \to A_s(g)$ such that the map $\widetilde{\varphi}\colon Z \to Z$ defined by sending  $z$ to $h(z).\psi(z)$ is a $k$-translation of $g$. To show $\widetilde{\varphi} \in \Aut_k(g)$, we need the following criteron. 
 
 \begin{proposition}\cite[Proposition 2.13]{GMVII} \label{criteron}
 Let $X$ be an  $s$-nilspace. Fix $0 \leq k \leq s+1$. Then a homeomorphism $\phi: X \to X$ is a $k$-translation if and only if for any $(s+1-k)$-cube $c$ of $X$ the configuration $\llcorner^k(c, \phi(c))$ is an $(s+1)$-cube.
 \end{proposition}
 
 Applying the above proposition to the cubespace $g^{-1}(y)$, we need to show that $\llcorner^k(c; \widetilde{\varphi}(c))$ is an $(s+1)$-cube of $g^{-1}(y)$ for every cube $c \in C^{s+1-k}(g^{-1}(y))$.
 
In order to measure the extent to which a configuration is failing to be a cube, let us introduce the definition of discrepancy in the setting of an $s$-fibration $g\colon Z \to Y$. Abbreviate $\pi_{g, s-1}$ by $\pi$. Let $c\colon \{0, 1\}^{s+1} \to Z$ be a map such that $\pi(c) \in C^{s+1}(\pi(Z))$, say, $\pi(c)=\pi(c_0)$ for some $c_0 \in C^{s+1}(Z)$. By the relative weak structure theorem (Theorem \ref{RWST}), there exists a unique map $\beta\colon \{0, 1\}^{s+1} \to A_s(g)$ such that $c=\beta . c_0$.

\begin{definition}
The {\bf discrepancy}\index{discrepancy} $\Delta(c)$\index{$\Delta(c)$} of $c$ is defined as
$$\Delta(c):=\sum_{\omega \in \{0, 1\}^{s+1}} (-1)^{|\omega|}\beta(\omega).$$
Here $|\omega|$ denotes the sum $\sum_{j=1}^{s+1}\omega_j$ for $\omega=(\omega_1, \ldots, \omega_{s+1})$.
\end{definition}
 Following the argument of \cite[Proposition 4.5]{GMVII}, we have that the discrepancy is well defined and $\Delta(c)=0$ if and only if $c \in C^{s+1}(Z)$.
We now  relativize the notion of cocycles and coboundaries of cubespaces: 
\begin{definition}
Let $f\colon X \to Y$ be a cubespace morphism and $A$ an abelian group. Fix an integer $\ell \geq 1$. Consider a continuous map $\rho \colon C^\ell_f (X) \to A$. We say $\rho$ is an {\bf $\ell$-coclycle\index{cocycle} on fibers of  $f$} if it is {\bf additive} in the sense that 
$$\rho([c_1, c_3])=\rho([c_1, c_2])+ \rho([c_2, c_3])$$
for any  $c_1, c_2, c_3 \in C^{\ell-1}_f(X)$ such that the three concatenations in the equation are  in $C^\ell_f(X)$.

We say $\rho$ is a {\bf coboundary}\index{coboundary} if there exists a continuous map $h: X \to A$ such that $\rho$ can be written as
$$\rho(c)=\partial^\ell h(c): =\sum_{\omega \in \{0, 1\}^\ell} (-1)^{|\omega|}h(c(\omega))$$
for every $c \in C^\ell_f(X)$.
\end{definition}

\begin{lemma} \label{repairing}
 Let $\psi\colon Z \to Z$ be a homeomorphism lifting some element $\varphi$ of $\Aut_k(g_{s-1})$. Then the map
$$\rho_\psi \colon C^{s+1-k}_g(Z) \to A_s(g)$$
sending $c$ to $\Delta(\llcorner^k(c; \psi(c)))$ is well defined. 
Moreover, $\psi$ can be repaired to a $k$-translation of $\Aut_k(g)$ if $\rho_\psi$ is an $(s+1-k)$-coboundary of some function with sufficient small norm.
\end{lemma}

\begin{proof}
Fix $y \in Y$ and a cube $c \in C^{s+1-k}(g^{-1}(y))$.  Since $\varphi \in \Aut_k(g_{s-1})$, from 
$$\pi(\llcorner^k(c; \psi(c)))=\llcorner^k(\pi(c); \pi(\psi(c)))=\llcorner^k(\pi(c); \varphi (\pi(c))),$$
 we have $\pi(\llcorner^k(c; \psi(c)))$ 
  is a cube of $C^{s+1-k}(\pi \circ g^{-1}(y))$. Thus the discrepancy $\Delta(\llcorner^k(c; \psi(c)))$ is well defined and so is $\rho_\psi$.

Assume $\rho_\psi =\partial^{s+1-k}f$ for some continuous function $f: Z \to A_s(g)$ with small norn. We will show it is possible to "integrate $f$ over $A_s(g)$" resulting with a function $F: Z \to A_s(g)$ such that $F(ax)=F(x)$ for every $a \in A_s(g)$ and $x \in Z$. In order to show that such a procedure is well-defined, let us recall a "lifting up \& down" technique for defining integration.
Since $A_s(g)$ is a compact abelian Lie group, there exists an isomorphism $\phi\colon A_s(g) \to (\bR/\bZ)^d\times K$ for some finite-dimensional torus $(\bR/\bZ)^d$ and some finite group $K$.  If $f$ is of sufficiently small norm, there exists a sufficient small $\delta$ such that the image of $f$ belongs to a $\delta$-ball $B$ around the identity with respect to a given compatible metric. Notice that $\phi$ induces an embedding $p: B \to \bR^d$.
For every $x \in X$, define $F(x)$ as
$$F(x):=p^{-1}\left(\int_{A_s(g)} p(f(ax))dm(a)\right),$$
where $m$ denotes the Haar measure on $A_s(g)$. It is easy to check that $F$ is well defined (see \cite[Subsection 5.2]{GMVII} for further explanation). 

Notice that  the repaired map $\widetilde{\varphi}(x):=F(x). \psi(x)$ is a bijection and hence is a homeomorphism.  By Lemma \ref{fiber of fibration}, $g^{-1}(y)$ is an $s$-nilspace.
Applying  Proposition \ref{criteron} to it,  we obtain that $\widetilde{\varphi}|_{g^{-1}(y)} \in \Aut_k(g^{-1}(y))$. Thus $\widetilde{\varphi}$ is a $k$-translation of
$\Aut_k(g)$.
\end{proof}

Let us introduce the notion of concatenation along the $k$-th axis.
\begin{definition}
 Let $1 \leq k \leq \ell+1$. Given a cubespace $X$ and two maps: $c_1, c_2\colon \{0, 1\}^\ell \to X$, the {\bf concatenation along the $k$-th axis} $[c_1, c_2]_k \index{$[c_1, c_2]$!$[c_1, c_2]_k$} \colon \{0, 1\}^{\ell+1} \to X$ is defined by sending $\omega$ to 
 $$c_1(\omega_1, \ldots, \omega_{k-1}, \omega_{k+1}, \ldots, \omega_{\ell+1})$$
 if $\omega_k=0$ and $c_2(\omega_1, \ldots, \omega_{k-1}, \omega_{k+1}, \ldots, \omega_{\ell+1})$ elsewhere.
\end{definition}
\noindent In particular, $[c_1, c_2]_{\ell+1}$ is simply the concatenation as defined previously.

We need a variant of \cite[Theorem 5.1]{GMVII} to prove Theorem \ref{openness}.

\begin{theorem} \label{fiber cocycle theorem}
Let $A$ be a compact abelian Lie group. Fix $s \geq 0$ and $\ell \geq 1$.
Then there exists  $\varepsilon > 0$ (depending only on $s, \ell$, and $A$) satisfying the following.

Let $\beta \colon Z \to Y$ be an  $s$-fibration.  Fix $0 < \delta < \varepsilon$.
Suppose that $\rho\colon C^\ell_\beta(Z) \to A$ is an $\ell$-cocycle on fibers of $\beta$ such that $d(\rho(c), \rho(c')) \leq \delta$
whenever $c, c'$ are cubes on the same fiber of $\beta$. Then
$$\rho=\partial^\ell f$$
for some continuous function $f: Z \to A$ which is almost constant on fibers of $\beta$, i.e. there exists a constant $c > 0$ (depending only on $s$ and $\ell$) such that $d(f(x), f(y)) \leq c\delta$ whenever $\beta(x)=\beta(y)$.
\end{theorem}

\begin{proof}
The proof uses a similar argument to the one used in the proof of \cite[Theorem 5.1]{GMVII}.
To modify the proof of \cite[Lemma 5.7]{GMVII}, consider the set 
$$T_1^\ell:=\{t: \{0, 1\}^{\ell-1} \to A_s(\beta):  [\overrightarrow{0}, t] \in C^\ell(\mathcal{D}_s(A_s(\beta)))\}.$$
Then for any $c \in C^{\ell-1}_\beta(Z)$ and $t \in T^\ell_1$,
we have $[c, t.c] \in C^\ell_\beta(Z)$. Therefore $\rho': C^{\ell -1}_\beta(Z) \to A$ sending $c$ to
$$\rho'(c):=\int_{T_1^\ell} \rho([c, t.c]_k)d\mu_{T_1^\ell}(t)$$
is well defined. Here $\mu_{T_1^\ell}$ denotes the  measure on $T_1^\ell$ induced from the Lebesgue measure (see \cite[Subsection 5.2]{GMVII} for more details). 

\end{proof}

\begin{proof}[Proof of Theorem \ref{openness}.]
Without loss of generality, we may assume $k \leq s$. Since $\pi_\ast$ is a group homomorphism, it suffices to show it is open at the identity. Let $\varphi\in\Aut_k(g_{s-1})$ be an element of small norm. We need to find a small $\widetilde{\varphi}$ in $\Aut_k(g)$ such that 
$$\pi_{g, s-1}\circ \widetilde{\varphi}=\varphi \circ \pi_{g, s-1}.$$ 

By Lemma \ref{preliftting}, we can lift $\varphi$ to a small homeomorphism $\psi\colon Z \to Z$ fixing fibers of $g$.  Since $\psi$ is of small norm, for any $c \in C^{s+1-k}_g(Z)$, $\llcorner^k(c, \psi(c))$ is close to the  $(s+1)$-cube $\Box^k(c) \in C^{s+1}(Z)$. Thus the discrepancy $\Delta(\llcorner^k(c, \psi(c)))$ is close to $\Delta(\Box^k(c))=0$,  where $\Box^k(c) \colon \{0, 1\}^{s+1} \to Z$ sends $\omega=(\omega_1, \ldots, \omega_{s+1})$ to $c(\omega_1, \ldots, \omega_{s+1-k})$. This implies that the image of the $(s+1-k)$-cocycle  $\rho_\psi$ has small diameter. Thus we can apply Theorem \ref{fiber cocycle theorem} to write $\rho_\psi$ as
$$\rho_\psi=\partial^{s+1-k}f$$
for some continuous map $f\colon  Z \to A_s(g)$ whose fibers are close to constant values. Then we can define $f$
as 
$$f=\widetilde{f}\circ g + h$$
for some continuous maps $\widetilde{f}\colon Y \to A_s(g)$ and $h\colon Z \to A_s(g)$ such that $h$ is almost constant $0_{A_s(g)}$ (in the sense of Theorem \ref{fiber cocycle theorem}).

Note that for any $c \in C^{s+1-k}_g(Z)$, $g(c)$ is a constant cube. It follows that $\rho_\psi=\partial^{s+1-k}h$.
 Since $\psi$ fixes the fibers of $\pi_{g, s-1}$ and $\varphi$ fixes the fibers of $g_{s-1}$, we have that $\psi$ fixes fibers of $g_{s-1}\circ \pi_{g, s-1}=g$. Since the image of $h$ is contained inside $A_s(g) \subseteq \Aut_s(g)$,  for each $z \in Z$, $h(z)$ also fixes fibers of $g$.  Therefore, we have
$$g(h(z).\psi(z))=g(\psi(z))=g(z),$$
i.e. $\widetilde{\varphi}:=h.\psi$ fixes the fibers of $g$. In summary, we prove that $\widetilde{\varphi} \in \Aut_k(g)$.

Finally, since $h$ is close to the  constant function $0_{A_s(g)}$, the repairing $\widetilde{\varphi}$ of small $\psi$ will also be small as desired.
\end{proof}

\section{Approximating by Lie-fibered fibrations}\label{sec:Approximating by Lie-fibered fibrations}

In this section, we complete the proof of Theorem \ref{relative inverse limit}. 
\subsection{The main steps of the proof}
Let $g\colon Z \to Y$ be an $s$-fibration between compact ergodic gluing cubespaces.
To deal with the general case as in \cite[Theorem 1.28]{GMVIII}, we need to first represent $Z$ as an inverse limit $\varprojlim Z_n$ respecting the fibers of $g$ as in \cite[Theorem 1.26]{GMVIII}, and then endow the fibers of the fibrations $Z_n \to Y$ with Host-Kra cube structure to obtain the cubespace isomorphism.

The following is an analogue of \cite[Theorem 1.26]{GMVIII} and covers the statements (1) and (2) of Theorem \ref{relative inverse limit}. 
\begin{theorem} \label{space as inverse limit}
Fix $s \geq 1$. Let $g\colon Z \to Y$ be an $s$-fibration between compact ergodic gluing cubespaces. Then $Z$ is isomorphic to an inverse limit 
$$\varprojlim Z_n$$
for an inverse system of $s$-fibrations between compact ergodic gluing cubespaces  $\{p_{m, n}\colon Z_n \to Z_m\}_{0\leq m \leq n \leq \infty}$ and compatible Lie-fibered $s$-fibrations $\{h_n\colon Z_n \to Y\}$. 
Here "compatible" means the following diagram commutes:
$$\xymatrix{
 Z \ar[r]  \ar[d]^g & \cdots \ar[r] & Z_n \ar[r]^{p_{n-1, n}} \ar[dll]_{h_n}^{\cdots} & Z_{n-1} \ar[r]  & \cdots \ar[r] & Z_1 \ar[dlllll]^{h_1}\\
 Y:=Z_0.
}$$
\end{theorem}

Let us start with some preliminary steps. Since every compact abelian group equals an inverse limit of compact abelian Lie groups \cite[Lemma 2.1]{GMVIII}, we can write the top structure group $A_s(g)$ of $g$ as an inverse limit of Lie groups $A_n$ with $A_0=\{0\}$, i.e.
$$A_s(g)=\varprojlim A_n.$$
Denote by $K_n$ the kernel of the quotient homomorphism $A_s(g) \to A_n$ and denote the orbit space of $Z$ under the action of by $K_n$, by $Z^{(n)}_\infty$, i.e. $Z^{(n)}_\infty:=Z/K_n$. 

\begin{lemma}
$Z^{(n)}_\infty$ has the gluing property.
\end{lemma}

\begin{proof}
Let $c_1, c_2, c_2', c_3 \in C^k(Z)$ such that $[c_1, c_2], [c_2', c_3] \in C^{k+1}(Z)$ and $\overline{c_2}=\overline{c_2'}$ in $C^k(Z^{(n)}_\infty)$.
We want to show $[\overline{c_1}, \overline{c_3}] \in C^k(Z^{(n)}_\infty)$.

\begin{proposition} \label{inner fibrant}
Let $f\colon X \to Y, g\colon Y \to Z$ and $h\colon X \to Z$ be three cubespace morphisms such that $h=g\circ f$. Suppose that h has $k$-completion and $g$ has $k$-uniqueness. Then $f$ has $k$-completion.
\end{proposition}

\begin{proof}
Assume that $\lambda$ is a $k$-corner of $X$ such that $f(\lambda)=c|_{\llcorner^k}$ for some $k$-cube $c$ of $Y$. We want to show there exists  $x \in f^{-1}(c(\overrightarrow{1}))$ completing $\lambda$ as a cube of $X$.

Note that $g(c)$ is a $k$-cube extending $(g\circ f)(\lambda)=h(\lambda)$. Since $h$ has $k$-completion, there exists  $x \in h^{-1}(g(c(\overrightarrow{1})))$ completing $\lambda$ to be a $k$-cube of $X$. Thus $(g\circ f)(x)=h(x)=g(c(\overrightarrow{1}))$. Thus  $f(x)$ completes $f(\lambda)$ as another $k$-cube sharing the same $g$-image as $c$. But since $g$ has $k$-uniqueness, it implies that $f(x)=c(\overrightarrow{1})$. This finishes the proof.
\end{proof}

Recall that for a compact abelian group $A$, $\cD_s(A)$ denotes the Host-Kra cubespace with respect to the $s$-filtration:
$$A=A_0=A_1=\cdots=A_s\supseteq A_{s+1}=\{0\}.$$
By the relative weak structure theorem (Theorem \ref{RWST}), there exists a unique $\alpha \in C^k(\cD_s(K_n))$ such that $c_2'=\alpha.c_2$.
Note that $[\alpha, \alpha] \in C^{k+1}(\cD_s(K_n))$. Thus $[\alpha.c_1, \alpha.c_2]$ is a $(k+1)$-cube of $X$. gluing with $[c_2', c_3]$, we obtainthat  $[\alpha.c_1, c_3] \in C^{k+1}(X)$. In particular, $[\overline{c_1}, \overline{c_3}] \in C^k(Z^{(n)}_\infty)$. Thus $Z^{(n)}_\infty$ has the gluing property.
\end{proof}

 Denote by 
$\beta_n\colon Z \to Z^{(n)}_\infty$ 
the quotient map. By the definition of $Z^{(n)}_\infty$, $g$ uniquely factors through $\beta_n$ via a map 
$g^{(n)}\colon Z^{(n)}_\infty \to Y.$ Since $g$ has $(s+1)$-uniqueness, applying the universal replacement property in a similar way to the proof of Proposition \ref{inducing $s$-fibration}, we have $g^{(n)}$ has $(s+1)$-uniqueness again. Since $g$ is a fibration, by Proposition \ref{inner fibrant}, $\beta_n\colon Z \to Z^{(n)}_\infty$ is fibrant for cubes of dimension greater than $s$.  Applying the universal replacement property of $Z$, $\beta$ is fibrant for cubes of dimension less than $s+1$.  Thus $\beta$ is a fibration. 
By Proposition \ref{universal property}, $g^{(n)}$ is again a fibration and hence is an $s$-fibration. Applying the $(s+1)$-uniqueness of $g$ again, we see that $\beta$ has $(s+1)$-uniqueness and hence is an $s$-fibration. Moreover, we have
$$\pi_{g^{(n)}, s-1}(Z^{(n)}_\infty)=\pi_{g, s-1}(Z)=Z^{(0)}_\infty.$$
In summary,  we  have the commutative diagram:
$$\xymatrixcolsep{5pc}\xymatrix{
Z \ar[d]_{\beta_n} \ar[ddr]^g & \\
Z^{(n)}_\infty \ar[d] \ar@{-->}[dr]^{g^{(n)}} & \\
\pi_{g, s-1}(Z) \ar[r]^{g_{s-1}}  & Y.
}$$

By the inductive hypothesis, the $(s-1)$-fibration $g_{s-1}\colon \pi_{g, s-1}(Z) \to Y $ factors as an inverse limit of a sequence of Lie-fibered $(s-1)$-fibrations 
$\psi_{0, m}\colon  Z^{(0)}_m \to Y$ along with compatible fibrations $\psi_{m, \infty}\colon \pi_{g, s-1}(Z) \to Z^{(0)}_m$. In fact, since $g_{s-1}$ is an $(s-1)$-fibration, so are $\psi_{0, m}$ and $\psi_{m, \infty}$. Here it is convenient to denote $Y$ by $Z^{(0)}_0$ and $Z$ by $Z^{(\infty)}_\infty$.

To factor $g^{(n)}$ properly based on the factorization of $g_{s-1}$, we will introduce some key definitions stemming from  the following lemma, which a variant version of Proposition \ref{canonical map}. 
\begin{lemma} \label{relative shadow}
Let $X, Y, Z$ be three compact ergodic gluing cubspaces and $\varphi: X \to Y$  a fibration. Let $g\colon X \to Z$ and $h: Y \to Z$ be two $s$-fibrations such that $g=h\circ \varphi$. Then there exists a unique fibration $\psi\colon \pi_{g, s-1}(X) \to \pi_{h, s-1}(Y)$ such that the following diagram 
$$\xymatrix{
X \ar[r]^\varphi \ar[d]_{\pi_{g, s-1}} \ar@/^3pc/[drr]^g  & Y \ar[dr]^h \ar[d]_{\pi_{h, s-1}} & \\
\pi_{g, s-1}(X) \ar@{-->}[r]^\psi \ar@/_2pc/[rr]^{g_{s-1}} & \pi_{h, s-1}(Y) \ar[r]^{h_{s-1}} & Z 
}$$
commutes.
\end{lemma}

\begin{proof}
Using $g=h\circ \varphi$, one can directly check that $\pi_{h, s-1}\circ \varphi$ factors through $\pi_{g, s-1}$ by a unique map $\psi$ such that $\pi_{h, s-1}\circ \varphi=\psi\circ \pi_{g, s-1}$. By the universal property of fibrations \cite[Lemma 7.8]{GMVI}, $\psi$ is a fibration. From
$$h_{s-1}\circ \psi \circ \pi_{g, s-1}=h_{s-1}\circ \pi_{h, s-1}\circ \varphi=h\circ \varphi=g=g_{s-1}\circ \pi_{g, s-1},$$
we obtain$h_{s-1}\circ \psi=g_{s-1}$. This shows that the diagram is commutative.
\end{proof}

\begin{definition}
With the setup in Lemma \ref{relative shadow}, we say that $\psi$ is the {\bf shadow}\index{shadow} of $\varphi$. Moreover, we say $\varphi$ is {\bf horizontal}\index{horizontal} if it satisfies one of the following equivalent conditions:
\begin{enumerate}
    \item $\varphi(x)\neq \varphi(x')$ for any $x \neq x' \in X$ with $x\sim_{g, s-1} x'$;\\
    \item for any $x \in X$,  the appropriate restriction of $\varphi$ induces a bijection between $\pi^{-1}_{g, s-1}(\pi_{g, s-1}(x))$ and $\pi^{-1}_{h, s-1}(\pi_{h, s-1}(\varphi(x)))$;
    \item the equivalence relation $\sim_{\varphi, s-1}$ is trivial.
\end{enumerate}

\end{definition}

Now we are ready to formulate the relative version of \cite[Proposition 2.5]{GMVIII}.

\begin{proposition} \label{middle fibration}
Let $Z^{(n)}_\infty, \beta_n, g^{(n)}, \psi_{m, \infty}$, etc. be as above. There exists a strictly increasing sequence $\M_1, \M_2, \ldots$ of positive integers satisfying the following. For each $n \in \bN$ and $m \geq \M_n$, there is a compact ergodic gluing cubespace $Z^{(n)}_m$ and an $s$-fibration 
$$h^{(n)}_m\colon Z^{(n)}_m \to Y$$
satisfying that:
\begin{enumerate}
    \item $\psi_{0, m}$ is the canonical $(s-1)$-th factor of $h^{(n)}_m$ with top structure group $A_n$;
    \item there is a  horizontal $s$-fibration $\varphi^{(n)}_m\colon Z^{(n)}_\infty \to Z^{(n)}_m$  such that $g^{(n)}=h_m^{(n)} \circ \varphi^{(n)}_m$ and $\psi_{m, \infty}$ is the relative shadow of $\varphi^{(n)}_m$;
    \item  If $m_1 \leq m_2$ and $n_1 \leq n_2$ are such that $Z^{(n_1)}_{m_1}$ and $Z^{(n_2)}_{m_2}$ are both defined, then the fibers of $\varphi^{(n_2)}_{m_2}\circ \beta_{n_2}$ refine the fibers of $\varphi^{(n_1)}_{m_1}\circ \beta_{n_1}$.
\end{enumerate}
In summary, we have  a  commutative diagram  
$$\xymatrixcolsep{5pc}\xymatrix{
Z^{(n)}_\infty \ar@{-->}[r]^{\varphi^{(n)}_m} \ar[d]_{A_n} \ar@/^4pc/[drr]^{g^{(n)}}  & Z^{(n)}_m \ar@{-->}[dr]^{h^{(n)}_m} \ar@{-->}[d]_{A_n} & \\
Z^{(0)}_\infty \ar[r]^{\psi_{m, \infty}} \ar@/_2pc/[rr]^{g_{s-1}} & Z^{(0)}_m \ar[r]^{\psi_{0, m}} & Y.
}$$
\end{proposition}

\begin{proof}[{\bf Proof of Theorem \ref{space as inverse limit}.}]
Using the notation of Proposition \ref{middle fibration}, we define $Z_n=Z^{(n)}_{\M_n}$ and the fibration $h_n=h^{(n)}_{\M_n}$. Note that the top structure group of $h_n$ is the Lie group $A_n$. By the induction hypothesis, the canonical $(s-1)$-factor $\psi_{0, \M_n}$ of $h_n$ is Lie-fibered. Combining these two facts, we have that $h_n$ is Lie-fibered. 

Define $p_{n, \infty}=\varphi^{(n)}_{\M_n}\circ \beta_n$. Since both $\varphi^{(n)}_{\M_n}$ and $\beta_n$ are $s$-fibrations, so is $p_{n, \infty}$. Then for every $n < \ell < \infty$, the fibers of $p_{\ell, \infty}$ refine the fibers of $p_{n, \infty}$. Thus by Proposition \ref{universal property}, $p_{n, \infty}$ and $p_{\ell, \infty}$ induce a unique fibration $p_{n, \ell}$ such that $p_{n, \infty}=p_{n, \ell}\circ p_{\ell, \infty}$. Moreover, we obtain a commutative diagram
$$ \xymatrixcolsep{5pc}\xymatrix{
 Z \ar[r]  \ar[d]^g  \ar[r]^{p_{\ell, \infty}} \ar@/^2pc/[rr]^{p_{n, \infty}} & Z_\ell \ar@{-->}[r]^{p_{n, \ell}} \ar[dl]_{h_\ell} & Z_n \ar[dll]^{h_n} \\
 Y.
}$$

For every $0 \leq n < \ell < o \leq \infty $, the condition  $p_{n, \ell}\circ p_{\ell, o}=p_{n, o}$ may be verified in a similar way to the absolute setting. We verify that the inverse system $Z_n$ separates points of $Z$. Let $z, z'$ be two distinct points of $Z$. If $\pi_{g, s-1}(z)\neq \pi_{g, s-1}(z')$, then 
$\psi_{\M_n, \infty}\circ\pi_{g, s-1}(z) \neq \psi_{\M_n, \infty}\circ \pi_{g, s-1}(z')$ as $n$ is large enough. Since $\psi_{\M_n, \infty}$ is the relative shadow of $\varphi^{(n)}_{\M_n}$, we have
$$\psi_{\M_n, \infty}\circ\pi_{g, s-1}=\pi_{h_n, s-1}\circ\varphi^{(n)}_{\M_n}\circ \beta_n=\pi_{h_n, s-1}\circ p_{n, \infty}.$$
It follows that $p_{n, \infty}(z)\neq p_{n, \infty}(z')$.

If $\pi_{g, s-1}(z) = \pi_{g, s-1}(z')$, then there is a unique $a \in A_s(g)$ such that $z'=az$. Note that $a \notin K_n$ as $n$ is large enough. Thus  $\beta_n(z)\neq \beta_n(z')$. Since $\varphi_{\M_n}^{(n)}$ is horizontal, we get
$$p_{n, \infty}(z)=\varphi^{(n)}_{\M_n}\circ \beta_n(z) \neq \varphi^{(n)}_{\M_n}\circ \beta_n(z')=p_{n, \infty}(z').$$

\end{proof}

The following is an analogue of \cite[Theorem 1.27]{GMVIII} \cite[Theorem 4]{ACS12}. We will give the proof in Section 5.3.
\begin{theorem} \label{endow cubestructure}
Let $X, Y, Z$ be three compact ergodic gluing cubespaces. Suppose that $\varphi\colon X \to Y$
is a fibration, and $g\colon X \to Z$ and $h\colon Y \to Z$ are two Lie-fibered $s$-fibrations  such that $g=h\circ \varphi$. Then for every $i\geq 1$, $\varphi$ induces a surjective continuous group homomorphism 
$$\Phi\colon \Aut_i^\circ(g) \to \Aut_i^\circ(h)$$ such that
$\Phi (u)\circ \varphi=\varphi \circ u$
for all $u \in \Aut_i^\circ(g) $. In summary,the following diagram is commutative:
$$\xymatrixcolsep{5pc}\xymatrix{
X \ar[r]^u \ar[d]_\varphi   & X \ar[dr]^g \ar[d]^\varphi & \\
Y \ar@{-->}[r]^{\Phi(u)}  & Y \ar[r]^h & Z.
}$$
\end{theorem}

\begin{proof}[{\bf Proof of Theorem \ref{relative inverse limit}}]
Statement $(1)$ and $(2)$ have been proven in Theorem \ref{space as inverse limit}. Let us prove statement $(3)$. The case $s=0$ is trivial since $g$ is clearly a cubespace isomorphism and each fiber of $g$ is simply a singleton. We now assume $s \geq 1$. By Theorem \ref{space as inverse limit}, since the diagram commutes, we can first define $g^{-1}(g(z))$ as an inverse limit of $h_n^{-1}(g(z))$ by restricting the projection map $Z \to Z_n$. By Theorem \ref{relative nilmanifold},  $h_n^{-1}(g(z))=h_n^{-1}(h_n(z_n))$ is isomorphic to $\Aut_1^\circ(h_n)/\Stab(z_n)$. Thus
$$g^{-1}g(z)\cong \varprojlim (h_n^{-1}h_n(z_n)) \cong \varprojlim (\Aut_1^\circ(h_n)/\Stab(z_n)). $$

To see that the above isomorphism is a  cubespace isomorphism, apply Theorem \ref{endow cubestructure} to the fibration $p_{n-1, n}$ and Lie-fibered $s$-fibrations $h_n$ and $h_{n-1}$, to obtain a surjective continuous homomorphism 
$$\Phi_{{n-1, n} }\colon \Aut_i^\circ(h_n) \to \Aut_i^\circ(h_{n-1}).$$
This induces an inverse limit $\varprojlim (\HK^k(\Aut_\bullet^\circ(h_n))/\Stab(z_n))$ for every $k \geq 0$. By Theorem \ref{relative nilmanifold}, $C^k(h_n^{-1}(h_n(z_n)))$ is isomorphic to $\HK^k(\Aut_\bullet^\circ(h_n))/\Stab(z_n)$. Thus we obtain the desired cubespace isomorphism.
\end{proof}

The following proposition gives a a useful condition for verifying when a nilspace fiber is strongly connected.

\begin{proposition} \label{prop:structure groups} Let $f \colon X \to Y$ be  a fibration of degree at most $d$ with structure groups $A_1,\ldots, A_d$,  then for all $y\in Y$,  the subcubespaces $f^{-1}(y)$ are nilspaces  of degree at most $d$ with structure groups $A_1,\ldots, A_d$.\end{proposition}
\begin{proof}Recall that for each $k \geq 0$ and $y \in Y$, $C^k(f^{-1}(y))$ is given by the restriction $C^k(X)\cap (f^{-1}(y))^{\{0,1\}^k}$. Since $f$ is a fibration, for each $k$-corner $\lambda$ of $f^{-1}(y)$, $f(\lambda)$ can be completed as a constant $k$-cube, we have $\lambda$ can be completed as a cube of $f^{-1}(y)$. Thus $f^{-1}(y)$ has $k$-completion. Since $f$ has $(d+1)$-uniqueness, it guarantees that $f^{-1}(y)$ has $(d+1)$-uniqueness. This proves that $f^{-1}(y)$ is a nilspace of degree at most $d$.
Now we show the top structure group of $f^{-1}(y)$ is $A_d$. The case  for other structure groups use the same argument which we omits. Set $A_d=A$ and $y=f(x)$ for some $x \in X$. By the construction of $A$ in the proof of \cite[Theorem 7.19]{GMVI}, for every $a \in A$, we have $f(ax)=f(x)$ and $ax\sim_{d-1} x$. Let $p \colon f^{-1}(y) \to f^{-1}(y)/\sim_{d-1}$ be the canonical projection map. It follows that $f^{-1}(y)$ is $A$-invariant and moreover $Ax \subseteq  p^{-1}(p(x))$. On the other hand, if $x' \in f^{-1}(y)$ and $x' \sim_{d-1} x$, by the relative weak structure theorem, there exists a unique $a \in A$ such that $x'=ax$. So we have $Ax=p^{-1}(p(x))$. Thus $A$ is the top structure group of $f^{-1}(y)$. \end{proof}

\subsection{Straight classes and sections}

To prove Proposition \ref{middle fibration}, we need an analogue of \cite[Proposition 2.13]{GMVIII}. Let us start with some preliminary steps. 
\begin{definition}
Let $X, Z, W$ be three compact ergodic gluing cubespaces.  Given an $s$-fibration $g\colon X \to Z$ and a fibration $\psi\colon \pi_{g, s-1}(X) \to W$, we call a subset $D \subseteq  X$ is a {\bf straight $\psi$-class}\index{ straight class} if there exists  $w \in W$ such that 
 \begin{enumerate}
     \item $D \cap \pi^{-1}_{g, s-1}(u)$ is a singleton for every $u \in \psi^{-1}(w)$ and $D$ is the union of those singletons;
     \item a configuration $c\colon \{0, 1\}^{s+1} \to D$ is a cube if and only if $\pi_{g, s-1}(c)$ is a cube of $W$.
 \end{enumerate}
 In short, $D$ is a $\pi_{g, s-1}$-lifting of some fiber of $\psi$ respecting cube structure.
 
Consider a configuration $c\colon \{0, 1\}^{s+1} \to X$ inducing a cube $\pi_{g, s-1}(c)$. In light of the relative weak structure theorem (Theorem \ref{RWST})  and \cite[Proposition 5.1]{GMVI}, there exists a unique element $a \in A_s(g)$ such that the application of $a$ to $c$ at $(0, \ldots, 0) \in \{0, 1\}^{s+1}$ results with a cube of $X$. We call such an element $a \in A_s(g)$ the {\bf discrepancy} of $c$ and denote it by $D(c)$.  

 Let $U \subseteq  W$ be an open subset. We say a continuous map $\sigma\colon \psi^{-1}(U) \to Z$ is a {\bf straight section}\index{ straight section} if 
 \begin{enumerate}
     \item $\pi_{g, s-1}\circ \sigma=\id_U$;
     \item for any $c_1, c_2 \in C^{s+1}(\psi^{-1}(U))$ with $\psi(c_1)=\psi(c_2)$, we have $D(\sigma(c_1))=D(\sigma(c_2))$.
\end{enumerate}

\end{definition}
We remark that the  straightness of a section $\sigma$ implies that $\sigma$ maps every fiber of $\psi$ onto a  straight $\psi$-class.

The following lemma is a relative version of \cite[Lemma 2.7]{GMVIII}.
\begin{lemma} \label{fiber class}
 Let $X, Y, Z$ be three compact ergodic gluing cubspaces and $\varphi: X \to Y$ be a fibration. Let $g\colon X \to Z$ and $h: Y \to Z$ be two $s$-fibrations such that $g=h\circ \varphi$. If $\varphi$ is horizontal, denoting by  $\psi$ the  shadow of $\varphi$, then each fiber of $\varphi$ is a  straight $\psi$-class. 
\end{lemma}

The following lemma is a relative analogue of \cite[Proposition 2.8]{GMVIII}, based on \cite[Theorem 1.25]{GMVIII} and the relative weak structure theorem (Theorem \ref{RWST}).
\begin{lemma} \label{enough straight classes}
Suppose that $g\colon X \to Z$ is an $s$-fibration between compact ergodic gluing cubespaces such that the top structure group is a Lie group $A$. Then for any $\varepsilon > 0$ there is  $\delta > 0$ satisfying the following property. 

Let $\psi\colon \pi_{g, s-1}(X) \to W$ be an $(s-1)$-fibration to another compact ergodic gluing cubespace $W$ such that $\psi$ is a $\delta$-embedding in the sense that every fiber of $\psi$ has diameter less than $\delta$. Then for every $c \in C^{s+1}(\pi_{g, s-1}(X))$, there is an open set $U \subseteq W$ satisfying that
\begin{enumerate}
    \item the image of $c$ is contained inside $ \psi^{-1}(U)$;
    \item there is a straight $\psi$-section $\sigma: \psi^{-1}(U) \to X$ with ${\rm diam} (\sigma (\psi^{-1}(b))) \leq \varepsilon$ for every $b \in U$.
\end{enumerate}
In particular, every $x \in X$ is contained in a  $\psi$-class  of small diameter (explicitly,  $ x \in a.\sigma\circ \psi^{-1}(\psi\circ \pi_{g, s-1}(x))$ for some $a \in A$ ).
\end{lemma}

The following is a relative analogue of \cite[Proposition 2.9]{GMVIII}.
\begin{lemma} \label{close class}
Let $g\colon X \to Z$ be an $s$-fibration  such that  the top structure group is a Lie group $A$. Then there exists  $\delta > 0$ depending only on $g$ satisfying the following. 

 Let  $\psi\colon \pi_{g, s-1}(X) \to W$ be an $(s-1)$-fibration for another compact ergodic gluing cubespace $W$. Suppose that $D_1$ and  $D_2$ are two  straight $\psi$-classes with the  same image of $\pi_{g, s-1}$ and the restriction $\pi_{g, s-1}|_{D_1\cup D_2}$ is a $\delta$-embedding. Then $D_1=aD_2$ for some $a \in A$.
\end{lemma}

\begin{proof}
Since $\psi$ is a fibration, by Proposition \ref{fiber of fibration}, the space $B:=\pi_{g, s-1}(D_1)$ is a compact ergodic nilspace. Applying \cite[Theorem 5.2]{GMVII} to the nilspace $B$, we obtain  $\varepsilon=\varepsilon(s, s+1, A)$ of \cite[Theorem 5.2]{GMVII}. Define $\delta=\varepsilon/2$.
  
  Denote by $\sigma_i\colon B \to D_i$ the inverses of $\pi_{g, s-1}$  restricted to $D_i$ for $i=1,2$. Let $f: B \to A$ be the continuous function determined by the equation $\sigma_2(y)=f(y).\sigma_1(y)$ for all $y \in B$. Since $D_1$ and $D_2$ are  straight classes, for every $c \in C^{s+1}(B)$, $\sigma_1(c)$ and $\sigma_2(c)$ are cubes of $X$. Thus by the relative weak structure theorem (Theorem \ref{RWST}), $\partial^{s+1}(f(c))=0$. Since $\pi_{g, s-1}|_{D_1\cup D_2}$ is a $\delta$-embedding, we have that the diameter of $\Ima (f)$ is less that $\varepsilon$. Applying \cite[Theorem 5.2]{GMVII}, $f$ is constant. Q.E.D. 
\end{proof}

\begin{remark}
We  note that the assumption on $\psi$ in Lemmas \ref{enough straight classes} and \ref{close class} is weaker than the corresponding assumption in the  absolute setting in \cite{GMVIII}.

\end{remark}

Combining Lemmas \ref{enough straight classes} with \ref{close class}, the straight $\psi$-classes induce an equivalence relation  which we denote by $\approx_\psi$\index{$\approx_\psi$}. This equivalence relation allows us to construct the desired cubespaces and fibrations stated in Proposition \ref{middle fibration}. 
\begin{lemma} \label{constructed space}
Let $g\colon X \to Z$ be an $s$-fibration  such that the top structure group is a Lie group $A$. Write $\pi=\pi_{g, s-1}$. Then there exists  $\delta >0$ depending only on $g$ satisfying the following. 

Let  $\psi\colon \pi_{g, s-1}(X) \to W$ be an $(s-1)$-fibration such that  $\psi$ is a $\delta$-embedding. Then the induced equivalence relation $\approx_\psi$ from the straight $\psi$-classes is closed. Moreover, let $u\colon W \to  Z$ be an $(s-1)$-fibrations with $g_{s-1}=u \circ \psi$. Then it holds:
\begin{enumerate}
    \item  the quotient map $\varphi\colon X \to X/\approx_\psi$ is a fibration and induces a fibration $\pi'\colon X/\approx_\psi \to W$ and an $s$-fibration $u':=u\circ \pi'$ such that $\varphi$ is  horizontal and the diagram below commutes
    $$\xymatrixcolsep{5pc}\xymatrix{
X  \ar@{-->}[r]^\varphi \ar[d]^A_{\pi} \ar@/^3pc/[drr]^g & X/\approx_\psi \ar@{-->}[d]^{\pi'} \ar@{-->}[dr]^{u'} \\
\pi_{g, s-1}(X) \ar[r]^\psi \ar@/_2pc/[rr]^{g_{s-1}} & W \ar[r]^u  & Z; 
}$$
    \item the top structure group of $u'$ is $A$ and $\pi'=\pi_{u', s-1}$;
\end{enumerate}

\end{lemma}

\begin{proof}
Based on Lemma \ref{close class}, the fact that the relation $\approx_\psi$ is closed follows from the argument of \cite[Proposition 2.13]{GMVIII}. It is clear from the definition of the equivalence relation that the induced map $\pi'$ is well-defined and $\psi \circ \pi=\pi'\circ \varphi$.  From $g_{s-1}=u\circ \psi$, we have
  $$g=g_{s-1}\circ \pi =u\circ \psi \circ \pi=u\circ \pi'\circ \varphi=u'\circ \varphi.$$
Thus $\psi$ is the  shadow of $\varphi$. 

We show $\varphi$ is a fibration. Since $g$ is an $s$-fibration and $u$ is an $(s-1)$-fibration, by a similar argument to the one in the proof of \cite[Lemma 2.14]{GMVIII}, we have that $u'$ has $(s+1)$-uniqueness. By Proposition \ref{inner fibrant}, $\varphi$ is fibrant for $k$-corners of every $k \geq s+1$.  From the fact that $\varphi$ is relatively $s$-ergodic (see Definition \ref{relative ergodic}) and the fact that $\psi$ is a fibration, it follows  that 
  $\varphi$ is fibrant for corners of lower dimension. Note that $A$ respects  straight $\psi$-classes. Thus $X/\approx_\psi$ inherits an $A$-action from the $A$-action on $X$. Moreover, the straightness guarantees that $W$ is exactly the orbit space induced and hence  $\pi'=\pi_{u', s-1}$ and $A$ is the top structure group of $u'$. 

\end{proof}

\begin{proof}[Proof of Proposition \ref{middle fibration}]
For each fixed $n$ applying Lemma \ref{constructed space} to the $s$-fibration $g^{(n)}: Z^{(n)}_\infty \to Y$, we obtain a $\delta_n$ satisfying the desired property. Take $M_n$ large enough such that for every $m \geq M_n$, $\psi_{m, \infty}$ is a $\delta_n$-embedding. 
Then we obtain the desired $s$-fibration $h^{(n)}_m \colon Z^{(n)}_m \to Y$ and a horizontal fibration $\varphi^{(n)}_m \colon Z^{(n)}_\infty \to Z^{(n)}_m$  such that $g^{(n)}=h^{(n)}_m \circ \varphi^{(n)}_m$ from Lemma \ref{constructed space}.

In light of Lemma \ref{close class}, following the argument as in the absolute setting, it holds that the fibers of $\varphi^{(n_2)}_{m_2}\circ \beta_{n_2}$ refine the fibers of $\varphi^{(n_1)}_{m_1}\circ \beta_{n_1}$ for $n_1\leq n_2, M_{n_1}\leq m_1\leq m_2$ with $M_{n_2} \leq m_2$.
\end{proof}

\subsection{Relations between relative translations}

 For an $s$-fibration $g\colon Z \to Y$ between compact ergodic gluing cubespaces, denote by $\Aut_i^\varepsilon(g)$\index{$\Aut(X)$!$\Aut_i^\varepsilon(g)$} the $\varepsilon$-neighborhood  of the identity in $\Aut_i(g)$ under the metric
 $$d(f,g):=\max_{x \in X} d(f(x), g(x)).$$
By Corollary \ref{Lieness}, $\Aut_i(g)$ is a Lie group. Thus $\Aut_i^\circ(g)$ is a group generated by $\Aut_i^\varepsilon(g)$ as $\varepsilon$ is small enough. Then Theorem \ref{endow cubestructure} is a consequence of the following general result.

\begin{theorem} \label{pushforward and backward}
Fix $i \geq 1$. Let $\varphi\colon X \to Y, g\colon X \to Z$ and $h\colon Y \to Z$ be three $s$-fibrations such that $g=h\circ \varphi$ and $g, h$ are Lie fibrations. 
Then for any $\varepsilon > 0$ there is  $\delta > 0$ satisfying the following property. For any $u \in \Aut_i^\delta(g)$ there is
 $u' \in \Aut_i^\varepsilon(h)$, and conversely for any $u' \in Aut_i^\delta(h)$ there is  $u \in \Aut_i^\varepsilon(g)$, such that 
$u'\circ \varphi=\varphi\circ u$.
\end{theorem}

We introduce vertical fibrations \cite[Definition 3.2]{GMVIII} as follows.
\begin{definition}
Let $\varphi\colon X \to Y, g\colon X \to Z$ and $h\colon Y \to Z$ be three $s$-fibrations  such that $g=h \circ \varphi$. We say $\varphi$ is a {\bf vertical fibration}\index{vertical fibration} if for any $x, x' \in X$ such that $\pi_{h, s-1}\circ \varphi(x)=\pi_{h, s-1} \circ \varphi(x')$, one obtains that $\pi_{g, s-1}(x)=\pi_{g, s-1}(x')$.
\end{definition}

We can factor a fibration in a relative way in contrast with \cite[Proposition 3.3]{GMVIII}. 
\begin{proposition} \label{fibration factorization}
Let $\varphi\colon X \to Y, g\colon X \to Z$ and $h\colon Y \to Z$ be three $s$-fibrations  such that $g=h \circ \varphi$. Then there exists a compact ergodic gluing cubespace $W$ and an $s$-fibration $k\colon W \to Z$ such that $\varphi$ factors as $$\varphi=\varphi_h \circ \varphi_v$$
for some vertical fibration $\varphi_v\colon X \to W$ (with respect to $g$ and $k$) and  horizontal fibration $\varphi_h\colon W \to Y$ (with respect to $k$ and $h$). In summary, the following diagram is commutative: 
$$\xymatrixcolsep{5pc}\xymatrix{
X \ar[dd]_{\varphi} \ar@{-->}[dr]^{\varphi_v} \ar@/^3pc/[ddrr]^g  & &\\
& W \ar@{-->}[dl]_{\varphi_h} \ar@{-->}[dr]^k & \\
Y  \ar[rr]^h &  & Z.
}$$
\end{proposition}

\begin{proof}
Define $W= X/\sim_{\varphi, s-1}$. Denote by $\varphi_v$ the quotient map $ X \to W$ and $\varphi_h$ the induced map $ W \to Y$. Define $k:=h \circ \varphi_h$. Then it is routine to check the desired properties by definition.
\end{proof}

The following lemma is an analogue of \cite[Proposition 3.5]{GMVIII}.
\begin{lemma} \label{pushforward criterion}
Let $\varphi \colon X \to Y, g\colon X \to Z$ and $h\colon Y \to Z$ be three $s$-fibrations  such that $g=h \circ \varphi$. Let $u \in \Aut_i(g)$. Suppose that the fibers of $\varphi$ refine the fibers of $\varphi \circ u$. Then there is a unique relative translation $u' \in \Aut_i(h)$ such that $u'\circ \varphi=\varphi \circ u$.
\end{lemma}

\begin{proof}
By Proposition \ref{universal property},  there exists a unique fibration $u'\colon Y \to Y$ such that $u'\circ \varphi=\varphi \circ u$. 
We check that $u' \in \Aut_i(h)$. 

Firstly,  since $u$ fixes the fibers of $g$, we have $u'$ fixes fibers of $h$. Let $n\geq i, y \in Y$,  and  $c \in C^n(h^{-1}(h(y)))$. Since fibrations are surjective \cite[Corollary 7.6]{GMVI}, we have that $c=\varphi(\widetilde{c})$ for some $\widetilde{c} \in C^n(X)$. Choose some $x \in X$ such that $\varphi(x)=y$. It follows that $\widetilde{c} \in C^n(g^{-1}(g(x)))$. Let $F \subseteq \{0, 1\}^n$ be a face of codimension $i$. Since $u \in \Aut_i(g)$, we have $[u]_F.\widetilde{c} \in C^n(g^{-1}(g(x)))$.  Thus $[u']_F.c=\varphi([u]_F.\widetilde{c})$ is a cube of $h^{-1}(h(y))$.
\end{proof}

The  following lemma is an analogue of \cite[Lemma 3.6]{GMVIII}. 
\begin{lemma} \label{vertical independence}
Let $\varphi\colon X \to Y, g\colon  X \to Z$ and $h\colon Y \to Z$ be three $s$-fibrations  such that $g=h\circ \varphi$ and $\varphi$ is  vertical.  Fix $u \in \Aut_1(g)$. Then for any $x, x' \in X$ with $\varphi(x)=\varphi(x')$, one has $\varphi \circ u(x)=\varphi \circ u(x')$.
\end{lemma}

\begin{proof}
Since $\varphi$ is  vertical, we have $x\sim_{s-1}x'$ and hence $\llcorner^s(x, x')$ is a cube. Moreover, since $u \in \Aut_1(g)$,  we obtain an $(s+1)$-cube $[\llcorner^s(x, x'), \llcorner^s(u(x), u(x'))]$.  Hence $c:=\varphi([\llcorner^s(x, x'), \llcorner^s(u(x), u(x'))])$ is also an $(s+1)$-cube. On the other hand, by ergodicity of $Y$, we have another $(s+1)$-cube $c':=\varphi([\Box^s(x), \Box^s(u(x))])$. Note that $c|_{\llcorner^{s+1}}=c'|_{\llcorner^{s+1}}$ and 
$h(c)=\Box^{s+1}(g(x))=h(c')$. By the $(s+1)$-uniqueness of $h$, it holds that 
$$\varphi(u(x))=c'(\overrightarrow{1})=c(\overrightarrow{1})=\varphi(u(x')).$$
\end{proof}

Combining Lemmas \ref{pushforward criterion} with \ref{vertical independence}, we complete the pushing-forward part of Theorem \ref{pushforward and backward} for vertical fibrations.
\begin{proposition} \label{vertical pushforward}
Let $\varphi\colon X \to Y, g \colon X \to Z$ and $h\colon Y \to Z$ be three $s$-fibrations  such that $g=h\circ \varphi$. Suppose that $\varphi$ is  vertical. 
 Then there is a continuous homomorphism $\Phi\colon \Aut_i(g) \to \Aut_i(h)$ such that 
$$\Phi(u) \circ \varphi =\varphi \circ u$$
for any $u \in \Aut_i(g)$.
\end{proposition}

Now we deal with the  horizontal case.
\begin{proposition} \label{horizontal pushforward}
Let $\varphi\colon X \to Y, g\colon X \to Z$ and $h\colon Y \to Z$ be three $s$-fibrations  such that $g=h\circ \varphi$. Suppose that $\varphi$ is  horizontal and $g, h$ are Lie fibrations. Then for any $\varepsilon > 0$ there exists  $\delta > 0$ satisfying the following property. For any $u \in \Aut_i^\delta(g)$ there is
 $u' \in \Aut_i^\varepsilon(h)$ such that $u'\circ \varphi=\varphi \circ u$.
\end{proposition}

\begin{proof}
By Proposition \ref{pushforward criterion}, it suffices to show as $\delta$ is small enough every map in  $ \Aut_i^\delta(g)$ preserves $\varphi$-fibers.  We induct on $s$ to prove this. 

The case $s=0$ is trivial as $\varphi$ is  a cubespace isomorphism. 
Denote by $\psi \colon \pi_{g, s-1}(X) \to \pi_{h, s-1}(Y)$ the shadow of $\varphi$. By the induction hypothesis, there exists $\delta_0 > 0$ such that any element of $ \Aut_i^{\delta_0}(g_{s-1})$ preserves $\psi$-fibers.  Let  $u \in \Aut_i^\delta(g)$ for $\delta$ to be decided later. By Proposition \ref{canonical map}, $u$ induces a map $v:=\pi_\ast(u) \in \Aut_i(g_{s-1})$. As $\delta$ is small enough, we have $v \in \Aut_i^{\delta_0}(g_{s-1})$. 
By the induction hypothesis, we have that  $v$ preserves $\psi$-fibers.

Let $y \in Y$ and fix a point $x_0 \in \varphi^{-1}(y)$. Define $z=\varphi (u(x_0))$. We show that $D_1:=u(\varphi^{-1}(y))=D_2:=\varphi^{-1}(z)$ by applying Proposition \ref{close class} to the Lie fibration $g$. Note that $u(x_0) \in D_1\cap D_2$.
  By Lemma \ref{fiber class}, $D_2$ is a  straight $\psi$-class. Since $v$ preserves $\psi$-fibers, we can deduce that $D_1$ is also a
 straight $\psi$-class. Denote by $\delta'$ the number $\delta$ in Proposition \ref{close class}. Finally, one can check that $\pi_{g, s-1}|_{D_1\cup D_2}$ is a desired $\delta'$-embedding  as $\delta$ is small. This will force that $D_1=D_2$.

\end{proof}

Following the argument of \cite[Proposition 3.9]{GMVIII}, based on Propositions \ref{vertical pushforward} and \ref{horizontal pushforward}, and  Lemmas \ref{identity orbit is open} and \ref{discreteness}, we obtain the following proposition.
\begin{proposition} \label{pushbackward}
Let $\varphi\colon X \to Y, g\colon X \to Z$ and $h\colon Y \to Z$ be three $s$-fibrations  such that $g=h \circ \varphi$ and $g, h$ are Lie fibrations. Then for any $\varepsilon > 0$ there exists  $\delta > 0$ satisfying the following property. For any $u' \in \Aut_i^\delta(h)$ there is
 $u \in \Aut_i^\varepsilon(g)$ such that $u' \circ \varphi=\varphi \circ u$. 
\end{proposition}

\begin{proof}[{\bf Proof of Theorem \ref{pushforward and backward}}]
By Proposition \ref{fibration factorization},  in order to prove the first statement, it is enough to consider the cases that $\varphi$ is horizontal and  vertical separately. These two cases are dealt with in 
Propositions \ref{vertical  pushforward} and \ref{horizontal pushforward} respectively. The second statement follows from  Proposition \ref{pushbackward}. 
\end{proof}

\section{Isomorphisms between fibers of a fibration}\label{sec:Isomorphisms between fibers of a fibration}

In this section we prove Theorem \ref{fiber isomorphism} which gives a natural condition under which the fibers of a Lie-fibered $s$-fibration are isomorphic as subcubespaces. 

Let us first recall the covering homotopy theorem \cite[Theorem 2.1]{C17} \cite[Theorem 11.3]{Ste51}.
\begin{theorem}[Covering homotopy theorem]\label{thm:covering_homotopy}
Let $p\colon E \to B$ and $q\colon Z \to Y$ be fiber bundles with the same fiber $F$, where $B$ is compact. Let $h_0$ be a bundle map 
$$
\xymatrix{
E \ar[r]^{\widetilde{h_0}} \ar[d]^p & Z  \ar[d]^q\\
B \ar[r]^{h_0} & Y.
}$$
Let $H\colon B\times I \to Y$ be a homotopy of $h_0$, i.e. $h_0=H|_{B \times \{0\}}$. Then there exists a covering $\widetilde{H}$ of the homotopy $H$ by a bundle map 
$$
\xymatrix{
E \times I \ar@{-->}[r]^{\widetilde{H}} \ar[d]^{p\times {\rm id}} & Z \ar[d]^q \\
B \times I \ar[r]^H & Y.
}$$
\end{theorem}

\begin{proof}[Proof of Theorem \ref{fiber isomorphism}]
We first modify the homotopy argument of \cite[Theorem 5.1]{R78} to obtain a local homeomorphism  and then repair it to a cubespace isomorphism. By compactness, the isomorphism can be built globally. 

Denote by $I$ the interval $[0, 1]$. Since $Y$ is path-connected, there exists a continuous map $H_0\colon \{y_0\}\times I \to Y$ such that $H_0(y_0, 0)=y_0$ and $H_0(y_0, 1)=y_1$. Define $R=\Ima(H_0)$ and $y_t=H_0(y_0, t)$.  For every $i=0, 1, \ldots, s-1$ abbreviate by $\pi_i$ the map $\pi_{g, i}\colon Z/\sim_{g, i+1} \to Z/\sim_{g,i}$. We first show for every $0 \leq i \leq s$ there exists an interval $I_i:=[0, t_i]$ for some $t_i=t_i(y_0)> 0$ and a continuous map 
$$G_i\colon g_i^{-1}(y_0)\times I_i \to g_i^{-1}(R)$$
such that 
\begin{enumerate}
    \item $G_i(x, 0)=x$ for all $x \in g_i^{-1}(y_0)$, i.e. $G_i$ is a homotopy of inclusion map;
    \item $G_i(\cdot, t)\colon g_i^{-1}(y_0) \to g_i^{-1}(y_t)$ is a cubespace isomorphism for every $t \in I_i$.
\end{enumerate}

The case $i=0$ is clear since the fibers of $g_0$ are singletons and $G_0:=H_0$ works.
Inductively, suppose that we have a continuous map $G_{s-1}\colon g_{s-1}^{-1}(y_0)\times I_{s-1} \to g_{s-1}^{-1}(R)$ with the desired properties. Consider the fiber bundle map given by the inclusion map
$$\xymatrix{
g^{-1}(y_0) \ar[d]^{\pi_{s-1}} \ar[r]^i & (g_{s-1})^{-1}(R) \ar[d]^{\pi_{s-1}} \\
(g_{s-1})^{-1}(y_0) \ar[r]^i & (g_{s-1})^{-1}(R)
}$$
Clearly $G_{s-1}$ is a homotopy  of the inclusion map $i: (g_{s-1})^{-1}(y_0) \to (g_{s-1})^{-1}(R)$.
Applying Theorem \ref{thm:covering_homotopy} to this bundle map, we obtain a fiber bundle map $H_s$ such that the following diagram 
$$\xymatrix{
g^{-1}(y_0)\times I_{s-1} \ar[d]^{\pi_{s-1}\times \id} \ar@{-->}[r]^{H_s} & g^{-1}(R) \ar[d]^{\pi_{s-1}} \\
(g_{s-1})^{-1}(y_0) \times I_{s-1} \ar[r]^{G_{s-1}} & (g_{s-1})^{-1}(R).
}$$
commutes and $H_s(x,0)=x$ for all $x \in g^{-1}(y_0)$. 
In particular, for every $t \in I_{s-1}$, $H_s$ restricts to a continuous map from $g^{-1}(y_0)\times\{t\}$ to $g^{-1}(y_t)$. 
As $\pi_{s-1}$ is a principal $A_s(g)$-bundle map, it holds that this restriction map is a homeomorphism.

By the induction hypothesis $G_{s-1}(\cdot, t)$ is a cubespace isomorphism for every $t \in I_{s-1}$. Thus the discrepancy 
$$\rho_{H_s(\cdot, t)} \colon C^{s+1}(g^{-1}(y_0)) \to A_s(g)$$
is well defined for every $t \in I_{s-1}$. Since $H_s$ is a homotopy, as $t$ is small enough, say, $t \in I_s:=[0, t_s]$ for some $0 < t_s \leq t_{s-1}$,  $H_s(\cdot, t)$ is of sufficiently small norm. Thus by \cite[Theorem 4.11]{GMVII} (or \cite[Lemma 3.19]{ACS12}), $\rho_{H_s(\cdot, t)}$ is a coboundary map. Furthermore, since $H_s$ is continuous, it can be  repaired to a (continuous) homotopy 
$$G_s: g^{-1}(y_0)\times I_s \to g^{-1}(R)$$
which restricts to a cubespace isomorphism from $g^{-1}(y_0)\times \{t\}$ to $g^{-1}(y_t)$ for every $t \in I_s$. 
This completes the inductive step.

Finally, for every $y \in R$, the above argument shows in particular that for every $t \in I$, there exists $\delta_t> 0$ such that $g^{-1}(y_u)$ is isomorphic to $g^{-1}(y_{u'})$ as cubespaces for every  $u, u'$ in the ball $(t-\delta_t, t+\delta_t)$. By compactness, we obtain a finite open cover of $I$ consisting of open balls such that the $g$-fibers over each ball are isomorphic as cubespaces. Then a finite composition of isomorphisms give an isomorphism between $g^{-1}(y_0)$ and $g^{-1}(y_1)$.
\end{proof}

\begin{remark}
In the statement of Theorem \ref{fiber isomorphism}, it is desirable to drop the assumption that $g$ is a Lie fibration. Let us explain the obstruction. Note that the proof of Theorem \ref{fiber isomorphism} is based on an induction with a finite number of steps. To deal with the general case, one may apply the factorization of $g$ as an inverse limit of Lie fibrations in Theorem \ref{relative inverse limit}. However, 
after an infinite steps of induction, one fails to obtain a homotopy for some strictly positive time interval. This gives rise to the obstruction to construct the desired cubespace isomorphism.
\end{remark}

\section{Factor maps between minimal distal systems are fibrations}\label{sec:Factor maps between minimal distal systems are fibrations}

In this section, we prove that every factor map between minimal distal systems is a fibration for the induced cubespace morphism. 

\begin{definition} \label{def:principal abelian group}
Let $(G, X)$ be a dynamical system and $K$ a compact group acting on $X$ such that the action of $K$ commutes with the action by $G$. We say a factor map $\pi: (G, X) \to (G, Y)$ is a {\bf (topological) group extension by $K$} if 
$$R_\pi:=\{(x, x') \in X^2: \pi(x)=\pi(x')\}=\{(x, kx): x \in X, k \in K\}.$$
If furthermore $K$ is an abelian group and $K$ acts on $X$ freely, we say $\pi$ is a {\bf principal abelian group extension}\footnote{Comparing to Definition \ref{def:principal fiber bundle}, it is easy to see that a group extension $\pi: (G, X) \to (G, Y)$  by a compact group $K$ which acts freely on $X$ is a $K$-principal bundle.}.
\end{definition}
\begin{lemma} \label{K invariant}
Let $(G, X)$ be a dynamical system and $\pi: X \to Y$  a group extension by a compact group $K$. Then 
for every $\ell \geq 0$ we have
$$K\NRP^{[\ell]}(X)=\NRP^{[\ell]}(X)$$
where  $K$ acts on $\NRP^{[\ell]}(X)$ by diagonal action.
\end{lemma}

\begin{proof}
Let $c \in C_G^\ell(X)$ and $k \in K$. We check that $kc \in C_G^\ell(X)$. By definition there exists $g_n \in \HK^\ell(G)$ 
and $x_n \in X$ such that $c=\lim g_n(x_n, x_n, \ldots, x_n)$. Thus 
$$kc=k(\lim g_n(x_n, x_n, \ldots, x_n))=\lim g_n(k(x_n, x_n, \ldots, x_n)) \in C_G^\ell(X).$$

Now let $(x, y) \in \NRP^{[\ell]}(X)$. By definition we have  $ (x, x, \ldots, x, y) \in C_G^{\ell+1}(X)$. It follows that $k(x, x. \ldots, x, y) \in C_G^{\ell+1}(X)$. In other words, $(kx, ky) \in \NRP^{[\ell]}(X)$.

\end{proof}

\begin{definition} \label{isometric}
A factor map $\pi: X \to Y$ is an {\bf isometric extension} if there exists a continuous function $d\colon R_\pi: \to \bR$ such that
\begin{enumerate}
    \item for every $y \in Y$ the restriction map $d|_{\pi^{-1}(y)\times \pi^{-1}(y)}$ is a metric on $\pi^{-1}(y)$;
    \item $d(gx, gx')=d(x, x')$ for every $g \in G$ and $(x, x') \in R_\pi$.
\end{enumerate}
\end{definition} 

  The following lemma says that every isometric extension between minimal systems factors through group extensions.

\begin{lemma} \cite[Page 15]{GlasnerB} \label{isometric extension} 
A factor map $\pi\colon X \to Y$ is an isometric extension of minimal systems if and only if there exists a compact group $K$  and a closed subgroup $H$ of $K$ such that $X$ admits a group extension $\widetilde{X}$ by  $H$ and $Y$ admits a group extension $\widetilde{X}$ by  $K$ such that the 
diagram 
$$\xymatrix{
\widetilde{X} \ar[r] \ar[dr] & X (\cong \widetilde{X}/H) \ar[d]^\pi  \\
                             & Y (\cong \widetilde{X}/K)
}$$
commutes.
\end{lemma}

The following proposition says that fibrations are stable under inverse limits.
\begin{proposition} \label{limit preserving}
Let $p_{i, i+1}\colon X_{i+1} \to X_i$ be a fibration for every $i=0, 1, 2, \ldots$.  Denote by $X$ the inverse limit of $X_i$. Then the induced map $f\colon X \to X_0$ is  a fibration. Similarly, the degree of fibrations is also preserved by the inverse limit operation.
\end{proposition}

\begin{proof}
Denote by $p_n$ the projection map $X \to X_n$. Assume that $\lambda$ is a $k$-corner of $X$ such that $f (\lambda)$ extends to a cube $c$ of $X_0$. We need to complete $\lambda$ as a cube  via some point $(x_i)_i$ of $X$  such that  $f((x_i)_i)=c(\overrightarrow{1})$.

 Note that $p_1(\lambda)$ is a $k$-corner of $X_1$ and $p_{0, 1}\circ p_1 (\lambda)= f (\lambda) $ extends to a cube of $X_0$ via $c(\overrightarrow{1})$. Since $p_{0, 1}$ is a fibration, there exists $u_1 \in (p_{0, 1})^{-1}(c(\overrightarrow{1}))$ completing $p_1 (\lambda)$ as a cube of $X_1$. 

Now $p_2(\lambda)$ is a $k$-corner of $X_2$ and $p_{1, 2}\circ p_2(\lambda)=p_1 (\lambda)$ extends to a cube of $X_1$ via $u_1$. Since $p_{1, 2}$ is a fibration, there exists $u_2 \in (p_{1,2})^{-1}(u_1)$ completing $p_2(\lambda)$ as a cube of $X_2$. Inductively, if $\alpha$ is a limit ordinal,  we obtaina point $(u_i)_{i < \alpha}$ in $X_\alpha$ such that 
$$p_{0, \alpha}((u_i)_{i< \alpha})=p_{0, 1}(u_1)=c(\overrightarrow{1}).$$
Here for every $i< \alpha$, $u_i$ completes $p_i(\lambda)=p_{i, \alpha}\circ p_\alpha (\lambda)$ as a cube of $X_i$. Thus $(u_i)_{i < \alpha}$ completes $p_\alpha (\lambda)$ as a cube of $X_\alpha$. 
We proceed by induction.
\end{proof}

\begin{proof} [Proof of Theorem \ref{dyn_factor_is_fibration}]
  The relative Furstenberg structure theorem states that a distal extension of minimal systems is given by a (countable) transfinite tower\footnote{This is defined in \cite[Appendix E14.3, E15.5]{VriesB}.} of isometric extensions \cite[Chapter V, Theorem 3.34]{VriesB}. Applying Proposition \ref{limit preserving}, we reduce to the case where $\pi$ is an isometric extension.
  By Lemma \ref{isometric extension} and Proposition \ref{universal property}, we can further reduce to the case where $\pi$ is a group extension by a compact group $K$.

Fix $k \geq 1$. Suppose that $\lambda$ is a $k$-corner of  $X$ such that $\pi(\lambda)$ can be extended to a cube $c$ of $Y$. Since $X$ is fibrant,  we can extend $\lambda$ to be a cube via some point $x_0$ in $X$. It follows that $\pi(x_0)$ and $c(\overrightarrow{1})$ are $(k-1)$-canonically related. Since $(G, X)$ is minimal, from \cite[Theorem 6.1]{GGY18},  we have ${\rm NRP}^{k-1}(Y)=(\pi\times \pi)({\rm NRP}^{k-1}(X))$.  Thus by Proposition \ref{alternative}, we have
$$(\pi(x_0), c(\overrightarrow{1})) \in {\rm NRP}^{k-1}(Y)=(\pi\times \pi)({\rm NRP}^{k-1}(X)).$$
Thus there exists $(x, z) \in {\rm NRP}^{k-1}(X)$ such that $\pi(x)=\pi(x_0)$ and $\pi(z)=c(\overrightarrow{1})$.

Now since $\pi$ is a group extension by $K$, there exists a unique $a \in K$ such that $x_0=ax$. By Proposition \ref{URP},  it suffices to show $(x_0, az) \in {\rm NRP}^{k-1}(X)$. Indeed, in such a case, $az$ completes $\lambda$ as a cube and 
$\pi(az)=\pi(z)=c(\overrightarrow{1})$.

By Lemma \ref{K invariant}, $K\NRP^{k-1}(X)=\NRP^{k-1}(X)$. Since $(x, z) \in \NRP^{k-1}(X)$, it follows that $(x_0, az) =a(x, z) \in \NRP^{k-1}(X)$.

\end{proof}

\section{Extensions of finite degree}\label{sec:extensions of finite degree}

In this section we investigate extensions of finite degree (see Definition \ref{def:extension of degree}).
\begin{proposition}\label{structure0}
Let $s \geq 1$ and $f \colon (G, X) \to (G, Y)$ an extension of degree at most $s$ such that $X$ is minimal distal. Then $f$ factors as a tower of principal abelian group extensions:
$$\xymatrix{
(G, X)  \ar[r] \ar[d]^f  & (G,  X/\NRP^{s-1}(f)) \ar[r] & \cdots \ar[r] & (G, X/\NRP^{[1]}(f)) \ar[dlll]  \\
(G, Y).
}$$
\end{proposition}

\begin{proof}
Since $X$ is minimal distal, by \cite[Theorem 7.10]{GGY18}, it is  fibrant. From Theorem \ref{dyn_factor_is_fibration}, $f$ is a fibration. By Proposition \ref{fibrant is gluing}, fibrant cubespaces have the gluing property. Thus applying Proposition \ref{inducing $s$-fibration}, we conclude that $f$ is an  $s$-fibration. Since $X$ is minimal, it is an ergodic cubespace and hence so are the induced quotient spaces. From Proposition \ref{alternative}, $\NRPs(f)=\sim_{f, s}$.  By Theorem \ref{RWST}, $f$ is factored as stated. It suffices to show every successive map in the tower is a principal abelian group extension. Let us show, for example,  $f_s: (G, X) \to (G,  X/\NRP^{s-1}(f))$ is a principal abelian group extension by the structure group $A_s$. 

Since $\NRPs(f)$ is a $G$-invariant closed equivalence relation, the quotient map $f_s$ is a factor map. By Theorem \ref{RWST}, $f_s$ is an $A_s$-principal fiber bundle. Thus we need only to show $A_s$ commutes with the $G$-actions. Recall that in  \cite[Page 48]{GGY18}, $A_s$ is constructed as
$$A_s=\NRP^{s-1}(f)/\sim_f,$$
where $(x, x')\sim_f (y, y')$ if and only if $f(x)=f(x'), f(y)=f(y')$ and $[\llcorner^s(x,x'), \llcorner^s(y, y')] \in C_G^{s+1}(X)$. Denote by $[x, x']_f$ the equivalence class of $(x, x')$.  Fix $x \in X, a \in A_s$ and $t \in G$. We need to show $a(tx)=t(ax)$. 
Set $x'=ax$. By definition \cite[Page 49]{GGY18}, $(x, x') \in \NRP^{s-1}(f)$ and $a=[x, x']_f$. Applying $(\Box^s(e), \Box^s(t)) \in \HK^{s+1}(G)$ to the $(s+1)$-cube $[\llcorner^s(x,x'), \llcorner^s(x,x')]$. We obtain another $(s+1)$-cube $[\llcorner^s(x,x'), \llcorner^s(tx,tx')]$. Note that $f(tx)=tf(x)=tf(x')=f(tx')$. It follows that 
$$a=[tx, tx']_f=[tx, t(ax)]_f.$$
In particular, by definition of $A_s$-actions, we obtain $a(tx)=t(ax)$.
\end{proof}

Recall the definition of  \emph{maximal $s$-fibration} in Proposition \ref{maximal fibration}. In \cite[Theorem 7.15]{GGY18}, it was shown that for a minimal system $(G, X)$, the maximal $s$-nilspace factor of $(X, C^\bullet_G(X))$ coincides with $X/\NRPs(X)$.  As a relative version of this, we have
\begin{proposition}\label{dynamical maximal}
Let $f\colon X \to Y$ be a factor map of minimal distal systems.  Then $g\colon X/\NRPs(f) \to Y$ is the maximal $s$-fibration for every $s\geq 0$ and $g$ is an extension of degree at most $s$ relative to $Y$.
\end{proposition} 

\begin{proof}
 From Proposition \ref{alternative}, we know $\NRPs(f)=\sim_{f, s}$.  By Proposition \ref{maximal fibration}, $g\colon X/\NRPs(f) \to Y$ is the maximal $s$-fibration. Moreover, $\NRPs(g)=\sim_{g,s}=\Delta$. Thus $g$ is an extension of degree at most $s$ relative to $Y$.
\end{proof}

In general,  given a factor map $f\colon X \to Y$ of minimal systems, from Proposition \ref{relative relation},  the induced map $g\colon X/\NRPk(f) \to Y$ is a distal extension  and has $(k+1)$-uniqueness.
\begin{question} \label{fibration question}
\begin{enumerate}
    \item Is $g$ a fibration?
    \item More generally, if $f$ is distal, can one conclude that $f$ is a fibration? In \cite[Example 3.10]{TY13}, Tu and Ye considered the projection map of the Denjoy minimal system onto the unit circle and showed that it is not a fibration. However it is also not distal.
\end{enumerate}
\end{question}

Recall that a minimal system $(G, X)$ is called a system of degree at most $s$ if $\NRPs(X)=\Delta$. In \cite[Theorem 7.14]{GGY18}, it is proved that $(G, X)$ is a system of degree at most $s$ if and only if it is an $s$-nilspace. 
As a relative analogue of this result, we remark that if the answer of Question \ref{fibration question} (1) is positive, then we will obtain a dynamical characterization of $s$-fibration. That is, for a factor map  $f: X \to Y$ of minimal systems, $f$ is an extension of degree at most $s$ relative to $Y$ if and only if $f$ is an $s$-fibration.

\begin{proof}[Proof of Theorem \ref{thm:structure_finite_degree_extension}]

By Proposition \ref{dynamical maximal}, $\pi$ is an $s$-fibration.
Applying Theorem \ref{relative inverse limit}, we obtain a
factorization of $\pi$
by $s$-fibrations $p_{m, n} \colon Z_n \to Z_m$ and Lie fibrations
$h_n \colon Z_n \to Y$ which are compatible with each other. It
suffices to show that every $p_{m, n}$ is a factor map. We induct on
the degree $s$ in order to prove this.

The case $s=0$ is trivial since then $\pi$ is an isomorphism. Assume
that the statement is true for $s-1$.
Recall that in the proof of Theorem \ref{relative inverse limit}, the
space $Z_n$ is constructed as a quotient space $X/\approx_{\psi_n}$
based
on a fibration map $\psi_n \colon \pi_{\pi, s-1}(X) \to B_n$ for some
cubespace $B_n$. To show $p_{m, n}$ is a factor map, it suffices to
show that the
equivalence relation $\approx_{\psi_n}$ is $G$-invariant. By the
induction hypothesis, $\psi_n$ is  $G$-equivariant. Let
$x\approx_{\psi_n} x'$ for some $x, x' \in X$. Then for every $s \in
G$
$$\psi(\pi_{\pi, s-1}(sx))=\psi(s\pi_{\pi, s-1}(x))=s\psi(\pi_{\pi,
s-1}(x))=s\psi(\pi_{\pi, s-1}(x')).$$
Thus $sx\approx_{\psi_n} sx'$.
\end{proof}

\bibliographystyle{alpha}

\bibliography{On_the_structure_theory_of_cubespace_fibrations}

\def\cprime{$'$} \def\cprime{$'$}
\begin{thebibliography}{GMV20b}


\bibitem[ACS12]{ACS12}
O. ~Antol\'{i}n Camarena and B. ~Szegedy. 
\newblock Nilspaces, nilmanifolds and their morphisms. 
\newblock {\em arXiv:1009.3825.}

\bibitem[Aus88]{A}
J.~Auslander.
\newblock {\em Minimal flows and their extensions}, volume 153 of {\em
  North-Holland Mathematics Studies}.
\newblock North-Holland Publishing Co., Amsterdam, 1988.
\newblock Notas de Matem\'atica [Mathematical Notes], 122.

\bibitem[Bro68]{Bro68}
I.U.~Bronstein.
\newblock  Equivalence relations in distal minimal transformation groups.
\newblock {\em Bul. Akad. \v{S}tiince RSS Moldoven.}, (1968), no.1, 51--59.

\bibitem[Bro]{bronshteuin1979extensions}
I.U.~Bronstein.
\newblock {\em Extensions of minimal transformation groups}.
\newblock Translated from the Russian. 
\newblock Martinus Nijhoff Publishers, The Hague, 1979.


\bibitem[Bro70]{bronshtein1970distal}
I.U.~Bronstein.
\newblock Distal extensions of minimal transformation groups.
\newblock {\em Sib. Math. J.}, {\bf 11} (1970), no. 6, 897--911.

\bibitem[Can17a]{candela2016cpt_notes}
P.~Candela.
\newblock Notes on compact nilspaces.
\newblock {\em Discrete Anal., 2017:16}, 2017.

\bibitem[Can17b]{candela2016alg_notes}
P.~Candela.
\newblock Notes on nilspaces: algebraic aspects.
\newblock {\em Discrete Anal.,  2017:15}, 2017.

\bibitem[CGSS19]{candela2019nilspace_morphisms}
P.~Candela, D.~Gonz{\'a}lez-S{\'a}nchez, and B.~ Szegedy.
\newblock On nilspace systems and their morphisms.
\newblock {\em Ergodic Theory Dynam. Systems}, {\bf 40} (2020), no. 11, 3015--3029.

\bibitem[CS18]{candela2018nilspace}
P.~Candela and B.~ Szegedy.
\newblock Nilspace factors for general uniformity seminorms, cubic
  exchangeability and limits.
 \newblock {\em Mem. Am. Math. Soc.} to appear. arXiv:1803.08758

\bibitem[C17]{C17}
R.~L.~Cohen. 
\newblock {\em The Topology of Fiber Bundles}, 
\newblock Lecture Notes. AMS Open Math Notes Series, 2017.

\bibitem[E68]{E68}
R. ~Ellis.
\newblock The structure of group-like extensions of minimal sets.
\newblock {\em Trans. Amer. Math. Soc.}, {\bf 134} (1968), 261--287.



\bibitem[EG60]{EG60}
R.~Ellis and W.H.~Gottschalk.
\newblock Homomorphisms of transformation groups.
\newblock {\em Trans. Amer. Math. Soc.}, {\bf 94} (1960), 258--271.

\bibitem[EK71]{ellis1971characterization}
R.~Ellis and H.~Keynes.
\newblock A characterization of the equicontinuous structure relation.
\newblock {\em Trans. Amer. Math. Soc.}, {\bf 161} (1971), 171--183.

\bibitem[Fur63]{furstenberg1963structure}
H.~Furstenberg.
\newblock The structure of distal flows.
\newblock {\em Amer. J. Math.}, {\bf 85} (1963), no. 3, 477--515.

\bibitem[Gla03]{GlasnerB}
E.~Glasner. 
\newblock {\em Ergodic theory via joinings}. 
\newblock Mathematical Surveys and Monographs, 101. 
\newblock American Mathematical Society, Providence, RI, 2003.

\bibitem[GGY18]{GGY18}
E.~Glasner, Y.~Gutman, and X.~Ye. 
\newblock Higher order regionally proximal equivalence relations for general minimal group actions. 
\newblock {\em Adv. Math.} {\bf 333} (2018), 1004--1041. 

\bibitem[Gle51]{Gle51}
A.~M.~Gleason. 
\newblock The structure of locally compact groups. 
\newblock {\em Duke Math. J.} {\bf 18} (1951), 85--104.

\bibitem[GTZ12]{GTZ12}
B.~Green, T.~Tao, and T.~Ziegler.
\newblock An inverse theorem for the {G}owers {$U^{s+1}[N]$}-norm.
\newblock {\em Ann. of Math. (2)}, {\bf 176} (2012), no.2,  1231--1372.

\bibitem[GL19]{gutman2019strictly}
Y.~Gutman and Z.~Lian. 
\newblock Strictly ergodic distal models and a new approach to the Host-Kra factors.	
\newblock {\em arXiv:1909.11349.}

\bibitem[GMV19]{GMVII}
Y.~Gutman, F.~Manners, and P.~Varj{\'u}.
\newblock The structure theory of nilspaces {II}: Representation as
  nilmanifolds.
\newblock {\em Trans. Amer. Math. Soc.},  {\bf 371}(2019), no.7,  4951--4992.

\bibitem[GMV20a]{GMVI}
Y.~Gutman, F.~Manners, and P.~Varj{\'u}.
\newblock The structure theory of nilspaces {I}.
\newblock {\em J. Anal. Math.},
{\bf 140} (2020), 299--369.

\bibitem[GMV20b]{GMVIII}
Y.~Gutman, F.~Manners, and P.~Varj{\'u}.
\newblock The structure theory of nilspaces {III}: Inverse limit
  representations and topological dynamics.
\newblock {\em Adv. Math.}, {\bf 365}:107059, 2020.



\bibitem[HK05]{HK05}
B.~Host and B.~ Kra.
\newblock Nonconventional ergodic averages and nilmanifolds.
\newblock {\em Ann. of Math. (2)}, {\bf 161} (2005), no.1, 397--488.

\bibitem[HK08]{HK08}
B.~Host and B.~Kra.
\newblock Parallelepipeds, nilpotent groups and {G}owers norms.
\newblock {\em Bull. Soc. Math. France}, {\bf 136} (2008), no.3, 405--437.

\bibitem[HK18]{host2018nilpotent}
B.~Host and B.~Kra.
\newblock {\em Nilpotent structures in ergodic theory}, volume 236.
\newblock American Mathematical Soc., 2018.

\bibitem[HKM10]{HKM10}
B.~Host, B.~Kra, and A.~Maass.
\newblock Nilsequences and a structure theorem for topological dynamical
  systems.
\newblock {\em Adv. Math.}, {\bf 224} (2010), no.1, 103--129.

\bibitem[Hus94]{H94}
D.~Husemoller.
\newblock {\em Fibre bundles}, volume~20 of {\em Graduate Texts in
  Mathematics}.
\newblock Springer-Verlag, New York, third edition, 1994.


\bibitem[McM78]{McM78}
 D.~McMahon. 
\newblock Relativized weak disjointness and relatively invariant measures. 
\newblock {\em Trans. Amer. Math. Soc.} {\bf 236} (1978), 225-237.

\bibitem[MW73]{MW73}
 D.~McMahon and T.-S.~Wu.
 \newblock Relative equicontinuity and its variations. 
 \newblock {\em Recent advances in topological dynamics (Proc. Conf., Yale Univ., New Haven, Conn., 1972)}, 201–205. 
 \newblock Lecture Notes in Math., Vol. 318, {\it Springer, Berlin}, 1973.
 
 \bibitem[P95]{P95}
D.~Penazzi. 
\newblock On the proximal and regionally proximal relation of an extension between minimal flows. 
\newblock {\em Topological dynamics and applications (Minneapolis, MN, 1995)}, 53--74,  Contemp. Math., {\bf 215}, Amer. Math. Soc., Providence, RI, 1998. 

\bibitem[R78]{R78}
M. ~Rees. 
\newblock On the fibres of a minimal distal extension of a transformation group. 
\newblock {\em Bull. London Math. Soc.} {\bf 10} (1978), no. 1, 97--104.

\bibitem[SY12]{SY12}
S.~Shao and X.~Ye.
\newblock Regionally proximal relation of order {$d$} is an equivalence one for
  minimal systems and a combinatorial consequence.
\newblock {\em Adv. Math.}, {\bf 231} (2012), no.(3-4), 1786--1817.

\bibitem[Ste51]{Ste51}
N. ~Steenrod. 
\newblock {\em The Topology of Fibre Bundles}, 
\newblock Princeton Mathematical Series, vol. 14, Princeton University Press, Princeton, N. J., 1951.

\bibitem[Sze12]{szegedy2012higher}
B.~Szegedy.
\newblock On higher order {F}ourier analysis.
\newblock {\em arXiv:1203.2260}.

\bibitem[Tao12]{tao2012higher}
Terence Tao.
\newblock {\em Higher order {F}ourier analysis}, volume 142.
\newblock American Mathematical Soc., 2012.

\bibitem[TY13]{TY13}
S.~Tu and X.~Ye. 
\newblock Dynamical parallelepipeds in minimal systems. 
\newblock {\em J. Dynam. Differential Equations} {\bf 25} (2013), no. 3, 765--776.

\bibitem[Vee68]{veech1968equicontinuous}
W.A.~Veech. 
\newblock The equicontinuous structure relation for minimal abelian transformation groups. 
\newblock {\em Amer. J. Math.} (1968),  723--732.


\bibitem[Vri93]{VriesB}
J.~de Vries. 
\newblock {\em Elements of Topological Dynamics}. 
\newblock Mathematics and its Applications, 257. Kluwer Academic Publishers Group, Dordrecht, 1993.

\end{thebibliography}

\vspace{0.5\textheight}

\printindex

\end{document}